\documentclass[11pt]{article}

\usepackage[square,numbers,sort&compress]{natbib}

\usepackage{tablefootnote}
\usepackage[utf8]{inputenc} 
\usepackage[T1]{fontenc}    
\usepackage{hyperref}       
\usepackage{url}            
\usepackage{booktabs}       
\usepackage{amsfonts}       
\usepackage{nicefrac}       
\usepackage{microtype}      
\usepackage{tikz}
\usetikzlibrary{decorations.pathreplacing,calc}

\usepackage[algo2e,ruled]{algorithm2e}
\LinesNumbered

\usepackage{verbatim}
\usepackage{amsthm}
\usepackage{amsfonts}
\usepackage[nonamelimits]{amsmath}
\usepackage{amssymb}
\usepackage{sansmath}
\usepackage{color}
\usepackage{xcolor}
\usepackage{bm}
\usepackage{bbm}
\usepackage{fullpage}
\usepackage[bottom]{footmisc}

\theoremstyle{plain}
\newtheorem{theorem}{Theorem}[section]

\newtheorem{lemma}[theorem]{Lemma}

\theoremstyle{definition}

\usepackage{subfig}
\usepackage{graphicx}
\usepackage{listings}
\usepackage{enumitem}
\usepackage{tabularx}
\newtheoremstyle{named}{}{}{\itshape}{}{\bfseries}{.}{.5em}{\thmnote{#3}}
\theoremstyle{named}

\usepackage{tikz}
\usetikzlibrary{fit,calc}

\usepackage{tikz}
\usetikzlibrary{cd}

\DeclareSymbolFont{extraup}{U}{zavm}{m}{n}
\DeclareMathSymbol{\varheart}{\mathalpha}{extraup}{86}
\DeclareMathSymbol{\vardiamond}{\mathalpha}{extraup}{87}

\usepackage{mathtools}

\makeatletter
\newcommand*{\inlineequation}[2][]{%
  \begingroup
    \refstepcounter{equation}%
    \ifx\\#1\\%
    \else
      \label{#1}%
    \fi
    \relpenalty=10000 %
    \binoppenalty=10000 %
    \ensuremath{%
      #2%
    }%
    ~\@eqnnum
  \endgroup
}
\makeatother
\usepackage[table]{xcolor}

\usepackage{threeparttable}

\usepackage{todonotes}

\usepackage{xspace}

\xspaceaddexceptions{[]\{\}}

\usepackage{hyperref}

\usepackage{thm-restate}

\usepackage{bm}
\makeatletter
\newcommand*{\rom}[1]{\expandafter\@slowromancap\romannumeral #1@}

\usepackage{multirow}

\usepackage{tablefootnote}
\usepackage{float}
\usepackage{makecell}
\usepackage{authblk}

\allowdisplaybreaks

\title{On the Complexity of Finding Stationary Points of $p$-th Order H\"older Smooth Functions in One Dimension}
\date{\vspace{-5ex}}
\begin{document}
\author[1]{Huanjian Zhou\thanks{zhou-huanjian185@g.ecc.u-tokyo.ac.jp}}
\author[1,2]{Taiji Suzuki\thanks{taiji@mist.i.u-tokyo.ac.jp }}
\affil[1]{The University of Tokyo}
\affil[2]{RIKEN Center for Advanced Intelligence Project}

\maketitle

\begin{abstract}%
We study the query complexity of finding an an $\varepsilon$-stationary points of a one-dimensional non-convex function whose $p$-th derivative is $\beta_p$-H\"older continuous. 
We determine the optimal complexity for deterministic and randomized algorithms and for oracles that return only derivatives, only function values, or full local information.
Randomness provably helps when only derivatives are available, function values help in all cases, and higher-order derivatives never do.
Together with known results in higher dimensions, this shows how these benefits depend on the dimension.
Our deterministic lower bound also implies that the optimal accuracy exponents
of objective-function-free adaptive regularization of order $p$
(OFFAR$p$) in every dimension $d\ge1$ and of deterministic
gradient-only trapping in dimension $d=2$.
\end{abstract}



\section{Introduction}

Smooth non-convex minimization problems  are common in machine learning,
statistics, and scientific computing \citep{BCN18,LW15,NW06}.
Since finding a global minimizer
is generally intractable without additional global structure~\citep{PV91}, the standard goal is to find an approximate first-order stationary point. We study the query complexity of this task and how it depends on the
smoothness of the objective, the information returned by the oracle, and the algorithm's access to randomness.

In high dimension, the $O(\varepsilon^{-2})$ rate of gradient descent is optimal
for $C^{1,1}$ objectives, even among randomized algorithms with access to
function values~\citep{carmon2017lower}.
The hard instances hide directions and require dimension polynomial in
$1/\varepsilon$, while faster algorithms exist in fixed
dimension~\citep{vavasis1993black,bubeck2020trap,hollender2023computational,zhou2025adaptive,ZSS26}.
We study the simplest such setting, dimension one, where there are no
directions to hide, so hardness can only come from what the oracle leaves
unobserved between queries.
This makes a complete characterization possible, and every lower bound extends
to all dimensions.

For one-dimensional $C^{1,1}$ objectives with normalized smoothness
constant and initial gap, the query complexity is completely
understood~\citep{chewi2023complexity}.
With derivative queries only, it is $\Theta(\varepsilon^{-2})$ for
deterministic algorithms and $\Theta(\varepsilon^{-1})$ for randomized
algorithms.
If function values are also available, it is
$\Theta(\log(1/\varepsilon))$ for both.
Hence randomness and function values each reduce the query complexity.

This classification is restricted to Lipschitz gradients and to oracles
that return at most first-order information.
In higher dimensions, both restrictions matter.
Higher-order smoothness improves the complexity of gradient-only
algorithms~\citep{agarwal2017finding, birgin2017worst,carmon2017convex,carmon2018accelerated, li2023restarted, marumo2024parameter, ZSS26}, and algorithms that query derivatives of order
up to $p$ improve it further~\citep{nesterov2006cubic,cartis2011adaptive,carmon2017lower,gratton2023convergence}.
In one dimension, it is not known how higher-order smoothness changes the
classification above, or whether higher-order derivatives help.
This motivates the following question:
\begin{center}
\emph{How do \underline{randomness}, \underline{function values}, and \underline{higher-order derivatives} help\\
under \underline{higher-order smoothness} in one dimension?}
\end{center}

\paragraph{Our results.} We answer this question by determining the optimal query complexity for deterministic and randomized algorithms with each oracle and smoothness (see Table~\ref{tab:complexity-comparison}).
Our main conclusions are as follows.

\begin{table}[ht]
\centering
\caption{Comparison of query complexities with normalized smoothness
constants and initial gap. Here $\nu_r:=r+\beta_r$.}
\label{tab:complexity-comparison}
\footnotesize
\renewcommand{\arraystretch}{1.1}
\setlength{\tabcolsep}{4pt}
\begin{tabular}{@{}llll@{}}
\toprule
Algorithm & Oracle & Smoothness & Complexity \\
\midrule

\multirow{7}{*}{Deterministic}
& $f'$
& $C^{1,1}$
& $\Theta(\varepsilon^{-2})$
  \citep{chewi2023complexity} \\

& $f'$
& $\displaystyle\bigcap_{r=1}^p C^{r,\beta_r}$
& $O\big(\min_{1\le r\le p}
  \varepsilon^{-\frac{\nu_r}{\nu_r-1}}\big)$~(Theorem~\ref{thm:randomized-derivative-upper-bound})\\

& $f',\ldots,f^{(p)}$
& $\displaystyle\bigcap_{r=1}^p C^{r,\beta_r}$
& $\Omega\big(\min_{1\le r\le p}
  \varepsilon^{-\frac{\nu_r}{\nu_r-1}}\big)$ (Theorem~\ref{thm:deterministic-derivative-jet-lower}) \\

& $f',\ldots,f^{(p)}$
& $C^{p,\beta_p}$
& $O\big(\varepsilon^{-\frac{\nu_p}{\nu_p-1}}\big)$
  (OFFAR$p$~\citep{gratton2021h}) \\
\midrule

\multirow{4}{*}{Randomized}
& $f'$
& $C^{1,1}$
& $\Theta(\varepsilon^{-1})$
  \citep{chewi2023complexity} \\

& $f'$
& $C^{p,\beta_p}$
& $O(\varepsilon^{-1})$ (Theorem~\ref{thm:randomized-derivative-upper-bound})\\

& $f',\ldots,f^{(p)}$
& $\displaystyle\bigcap_{r=1}^p C^{r,\beta_r}$
& $\Omega(\varepsilon^{-1})$ (Theorem~\ref{thm:randomized-derivative-lower-bound}) \\
\midrule

\multirow{4}{*}{\shortstack[l]{Deterministic \\or randomized}}
& $f,f'$
& $C^{1,1}$
& $\Theta(\log(\frac{1}{\varepsilon}))$
  \cite{chewi2023complexity} \\

& $f$
& $C^{p,\beta_p}$
& $O(\log(\frac{1}{\varepsilon}))$ (Theorem~\ref{thm:deterministic-zeroth-order-upper}) \\

& $f,f',\ldots,f^{(p)}$
& $\displaystyle\bigcap_{r=1}^p C^{r,\beta_r}$
& $\Omega(\log(\frac{1}{\varepsilon}))$ (Theorem~\ref{thm:randomized-full-jet-lower-bound})\\
\bottomrule
\end{tabular}
\vspace{-0.2cm}
\end{table}
\begin{enumerate}[leftmargin = *,itemsep=2pt, parsep=0pt, topsep=3pt]
\item \textbf{Optimality of OFFAR$p$ for any dimension $d \geq 1$.} Our deterministic lower bound can be extended to every dimension $d\ge1$ by
embedding.
It shows that the $\varepsilon^{-\frac{\nu_p}{\nu_p-1}}$ dependence on
$\varepsilon$ of the H\"older extension of objective-function-free adaptive
regularization of order $p$ (OFFAR$p$)~\citep{gratton2023convergence}
is optimal among deterministic algorithms that query only derivatives of
orders $1,\ldots,p$.
For $\beta_p=1$, the classical matching oracle lower bound uses
a construction with $d=\Omega_p(\varepsilon^{-\frac{p+1}{p}})$~\citep{carmon2017lower}. This lower bound also implies that deterministic gradient-only trapping is optimal in dimension $d=2$~\citep{ZSS26}.
%

\item \textbf{The benefit of randomness for gradient-only algorithms.} 
Under higher-order H\"older smoothness, randomization reduces the
optimal gradient-only query complexity in one dimension from
$\Theta(\varepsilon^{-\frac{\nu_p}{\nu_p-1}})$
to $\Theta(\varepsilon^{-1})$.
The deterministic and the randomized lower bounds both extend to every dimension $d\geq 1$ by embedding.
In dimension two, the randomized gradient-only trapping algorithm
of \cite{ZSS26} makes
$\widetilde O(\varepsilon^{-\frac{2\nu_p}{2\nu_p-1}})$ queries,
stricyly fewer than the $\Omega(\varepsilon^{-\frac{\nu_p}{\nu_p-1}})$ required of
deterministic algorithms, so randomness still helps in two dimension.
By contrast, for $C^{1,1}\cap C^{2,1}$ objectives in sufficiently
high dimensions, \cite{ZSS26} prove a randomized lower bound
of $\Omega(\varepsilon^{-\frac74})$,
matching the deterministic gradient-only upper bound.
%

\item \textbf{Function values help but higher-order derivatives do not.}
For a fixed smoothness class in one dimension, querying
$f',\ldots,f^{(p)}$ does not improve the optimal complexity
over querying $f'$ alone.
By contrast, function values alone achieve $\Theta(\log(1/\varepsilon))$
complexity for both deterministic and randomized algorithms,
and this rate remains optimal even with access to
$f,f',\ldots,f^{(p)}$.
Function values also help deterministic algorithms in dimension two:
for $C^{1,1}$ objectives, the optimal function-value-only and
gradient-only complexities are $\Theta(\varepsilon^{-1})$ and
$\Theta(\varepsilon^{-2})$,
respectively~\citep{hollender2023computational,chewi2023complexity}.

By contrast, higher-order derivatives can improve complexity
in sufficiently high dimensions.
For $C^{1,1}\cap C^{2,1}$ objectives, randomized algorithms
with access to function values and gradients require
$\Omega(\varepsilon^{-\frac74})$ queries~\citep{ZSS26},
whereas algorithms with access to gradients and Hessians achieve
$O(\varepsilon^{-\frac32})$ queries~\citep{gratton2023convergence}.
Thus, the benefits of function values and higher-order derivatives
depend on dimension.

\item \textbf{The benefit of high-order smoothness.}
In one dimension, smoothness affects only the deterministic derivative-only
complexity of $\Theta(\varepsilon^{-\frac{\nu_p}{\nu_p-1}})$. Its exponent $\frac{\nu_p}{\nu_p-1}$ decreases toward $1$ as
$\nu_p=p+\beta_p$ grows, so the gap to randomized algorithms narrows as the
order of smoothness increases.
The randomized derivative-only complexity $\Theta(\varepsilon^{-1})$ and
the complexity $\Theta(\log(1/\varepsilon))$ with function values do not
depend on $\nu_p$.
\end{enumerate}

\section{Preliminaries}

\paragraph{Function classes and stationarity.}
Fix an integer $p\ge1$ and a set
$\mathcal I\subseteq\{1,\ldots,p\}$ with $p\in\mathcal I$.
For each $r\in\mathcal I$, let $\beta_r\in(0,1]$ and $L_r>0$,
and write $\boldsymbol\beta:=(\beta_r)_{r\in\mathcal I}$,
$\mathbf L:=(L_r)_{r\in\mathcal I}$, and $\nu_r:=r+\beta_r$.
For $f\in C^p(\mathbb R)$, define
$[f^{(r)}]_{\beta_r}
:=\sup_{x\ne y}
\frac{|f^{(r)}(x)-f^{(r)}(y)|}{|x-y|^{\beta_r}}$
and $f_\star:=\inf_{x\in\mathbb R}f(x)$.
Given $\Delta>0$, let
\[
\mathcal F_{\mathcal I,\boldsymbol\beta}(\mathbf L,\Delta)
:=
\left\{
f\in C^p(\mathbb R):
\begin{array}{l}
[f^{(r)}]_{\beta_r}\le L_r
\quad\text{for every }r\in\mathcal I,\\
f_\star>-\infty,\quad f(0)-f_\star\le\Delta
\end{array}
\right\}.
\]
The choices $\mathcal I=\{p\}$, $\mathcal I=\{1,p\}$,
and $\mathcal I=\{1,\ldots,p\}$ impose regularity at the highest
order, at the first and highest orders, and at every order,
respectively. In particular, $\mathcal I=\{1,p\}$ with $\beta_1=1$
describes the class with both Lipschitz gradients and
$\beta_p$-H\"older continuous $p$-th derivatives.
A point $x\in\mathbb R$ is
$\varepsilon$-stationary if $|f'(x)|<\varepsilon$.

\paragraph{Oracle models.}
We consider exact local oracles queried at arbitrary points
$x\in\mathbb R$. The zeroth-order oracle returns
$\mathcal O_f^{0}(x):=f(x)$, and the derivative-only oracle returns
$\mathcal O_f^{1}(x):=f'(x)$.
The value-gradient oracle returns
$\mathcal O_f^{0:1}(x):=(f(x),f'(x))$.
To study access to higher-order derivatives, we also define
$\mathcal O_f^{1:p}(x):=(f'(x),\ldots,f^{(p)}(x))$
and
$\mathcal O_f^{0:p}(x):=(f(x),f'(x),\ldots,f^{(p)}(x))$, which we call the full-information oracle.
Thus, $\mathcal O_f^{1:p}$ provides derivatives without function
values, whereas $\mathcal O_f^{0:p}$ additionally provides the
function value. Each query returns the entire specified tuple
and counts as one oracle call. 

\paragraph{Algorithms.}
Given $p,\mathcal I,\boldsymbol\beta,\mathbf L,\Delta$ and
$\varepsilon$, an algorithm $A$ adaptively queries $\mathcal O_f$
and outputs $\widehat x_A(f)\in\mathbb R$.
Its queries and output depend on the preceding oracle responses
and, if randomized, an internal random seed independent of $f$.
The algorithm succeeds if $|f'(\widehat x_A(f))|<\varepsilon$.

\paragraph{Query complexity.}
For $*\in\{\mathrm{det},\mathrm{rand}\}$, let
$\mathcal C_*^{(\mathcal I,\boldsymbol\beta)}
(\varepsilon;\mathbf L,\Delta,\mathcal O)$
be the smallest $N\in\mathbb N_0$ such that some algorithm of type $*$
makes at most $N$ oracle queries on every run and finds an
$\varepsilon$-stationary point for every
$f\in\mathcal F_{\mathcal I,\boldsymbol\beta}(\mathbf L,\Delta)$,
with certainty if deterministic and with probability at least
$1/2$ if randomized.
If no such finite $N$ exists, the complexity is defined to be $+\infty$.

The following rescaling normalizes the highest-order H\"older
constant and the initial value gap as $L_p=\Delta=1$, while transforming the
remaining regularity constants. The proof is deferred to Appendix~\ref{app:1}.

\begin{lemma}[\bf Normalization]
\label{lmm:normal}
Let $\nu_p:=p+\beta_p$ and set $s:=\big(\frac{\Delta}{L_p}\big)^{\frac1\nu},
\widehat L_r:=\frac{L_r s^{r+\beta_r}}{\Delta}
\quad(r\in\mathcal I),
\widehat{\mathbf L}:=(\widehat L_r)_{r\in\mathcal I}.$
Then $\widehat L_p=1$, and for any of the oracle models defined
above, $*\in\{\mathrm{det},\mathrm{rand}\}$, and $\varepsilon>0$,
\[
\mathcal C_*^{(\mathcal I,\boldsymbol\beta)}
(\varepsilon;\mathbf L,\Delta,\mathcal O)
=
\mathcal C_*^{(\mathcal I,\boldsymbol\beta)}
\big(
    \tfrac{\varepsilon}
    {L_p^{1/\nu_p}\Delta^{(\nu_p-1)/\nu_p}};
    \widehat{\mathbf L},1,\mathcal O
\big).
\]
\end{lemma}

\section{Lower Bound for Randomized Derivative-Only Algorithms}
\label{sec:randomized-lower-bound}
In this section, we show that randomized derivative-only algorithms require
$\Omega(\varepsilon^{-1})$ queries.
The bound holds even when H\"older conditions are imposed at several orders
simultaneously, and even when each query returns all derivatives up to
order $p$.

\begin{theorem}[\bf Randomized derivative-only lower bound]
\label{thm:randomized-derivative-lower-bound}
There exists a constant $c_p>0$, depending only on $p$, such that
for every
$0<\varepsilon\le
c_p\min_{r\in\mathcal I}
L_r^{\frac{1}{\nu_r}}
\Delta^{\frac{\nu_r-1}{\nu_r}}$,
\[
\mathcal C_{\mathrm{rand}}^{(\mathcal I,\boldsymbol\beta)}
(\varepsilon;\mathbf L,\Delta,\mathcal O^{1:p})
\ge
c_p
\min_{r\in\mathcal I}
L_r^{\frac{1}{\nu_r}}
\Delta^{\frac{\nu_r-1}{\nu_r}}
\varepsilon^{-1}.
\]
\end{theorem}
The minimum selects the most restrictive of the H\"older conditions.
Equivalently, one may take the minimum over all $1\le r\le p$ with the
convention $L_r=+\infty$ for $r\notin\mathcal I$, since such terms do not
affect the minimum.
The exponent of $\varepsilon$ does not depend on $p$ or
$\boldsymbol\beta$, and the bound holds even though each query returns all
derivatives up to order $p$.
Theorem~\ref{thm:randomized-derivative-upper-bound} gives a matching upper
bound that queries $f'$ only.

\subsection{Hard Family}
The construction hides the stationary points in one of
$m=\Theta_p(\varepsilon^{-1})$ candidate regions.
Outside the hidden region, all instances return the same derivatives.
A query can therefore distinguish at most one previously untested
candidate from the others.
The final output can identify at most one additional candidate,
even if it was never queried.
Consequently, an algorithm using $N$ queries has average success
probability at most $(N+1)/m$.
The regularity constraints are enforced by a fixed polynomial
transition whose derivatives match at the boundaries.

We construct instances in
$\mathcal F_{\mathcal I,\boldsymbol\beta}(\mathbf1,1)$.
Define $\eta_p(t):=
\frac{\int_0^t u^p(1-u)^p\,du}
     {\int_0^1 u^p(1-u)^p\,du},
0\le t\le1,$
and set
$K:=2\max\{1,\max_{1\le r\le p+1}
\|\eta_p^{(r)}\|_{\infty,[0,1]}\}$.
Extend $\eta_p$ by $0$ for $t<0$ and by $1$ for $t>1$.
Then $\eta_p\in C^p(\mathbb R)$, $0\le\eta_p\le1$, and
$\eta_p^{(r)}(0)=\eta_p^{(r)}(1)=0$ for $1\le r\le p$.

For $0<\varepsilon\le(8K)^{-1}$, let
$m:=\lfloor(2K\varepsilon)^{-1}\rfloor$ and $\delta:=1/m$,
so that $m\ge4$ and $\delta\ge2K\varepsilon$.
For each $j\in\{1,\ldots,m\}$, define
\[
f_j(x):=\frac{1-\delta x+\eta_p(x-j+1)}{K},
\qquad 0\le x\le m.
\]
Extend $f_j$ periodically to $[0,\infty)$ with period $m$,
and set $f_j(x):=(1-\delta x)/K$ for $x<0$.
Define the hidden region
\[
H_j:=\bigcup_{k\ge0}(km+j-1,km+j).
\]
Outside $H_j$, we have
$f_j'=-\delta/K\le-2\varepsilon$ and
$f_j^{(r)}=0$ for $2\le r\le p$.
Thus every $\varepsilon$-stationary point of $f_j$ lies in $H_j$,
and all derivative responses outside the hidden region
are independent of $j$.

Figure~\ref{fig:hard-family-local} illustrates $f_j$ and its
derivative near the hidden interval $(j-1,j)$.
The marked points are stationary points of $f_j$; outside $H_j$,
the derivative is constant at $-\delta/K$.

\begin{figure}[t]
    \centering
    \includegraphics[width=\textwidth]{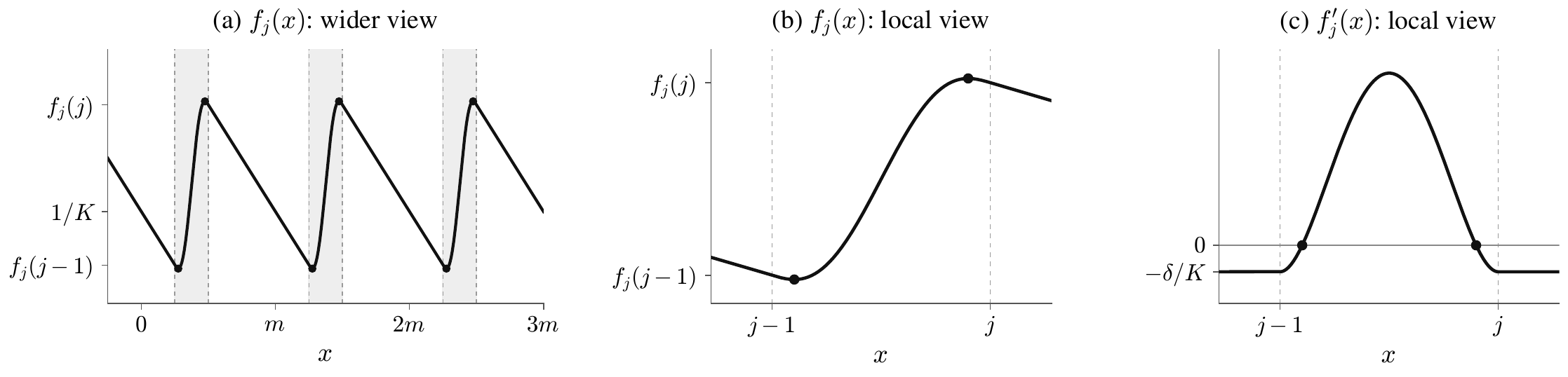}
    \vspace{-0.8cm}
    \caption{The hard instance $f_j$ on a wider interval (left), a local view
near the hidden interval $(j-1,j)$ (middle), and its derivative
near the same interval (right), illustrated for
$p=2$, $m=4$, and $j=2$.
The the gray shaded regions mark the hidden intervals.
Dashed lines mark the interval endpoints.
Dots mark the stationary points and the corresponding
derivative zeros. Outside $H_j$, $f_j'=-\delta/K$.}
    \label{fig:hard-family-local}
        \vspace{-0.8cm}
\end{figure}

\subsection{Properties and Proof Sketch}
The following lemma summarizes the key properties of the hard family.
Its proof is deferred to Appendix~\ref{app:2}.

\begin{lemma}[\bf Properties of the hard family]
\label{lem:hard-family-properties}
For $0<\varepsilon\le(8K)^{-1}$, the preceding construction
satisfies $m\ge4$ and $m\ge(4K\varepsilon)^{-1}$.
Moreover:
\begin{enumerate}[leftmargin = *,itemsep=2pt, parsep=0pt, topsep=3pt]
    \item \textbf{Regularity and initial gap.}
    For every $j\in[m]$,
    $f_j\in\mathcal F_{\mathcal I,\boldsymbol\beta}(\mathbf1,1)$.

    \item \textbf{Indistinguishability.}
    For every $j\in[m]$ and $x\notin H_j$,
    $\mathcal O_{f_j}^{1:p}(x)=(-\delta/K,0,\ldots,0)$,
    where $\delta/K\ge2\varepsilon$.

    \item \textbf{Localization of stationary points.}
    The sets $H_1,\ldots,H_m$ are pairwise disjoint.
    Every $\varepsilon$-stationary point of $f_j$ belongs to $H_j$,
    and each connected component of $H_j$ contains a stationary point.
\end{enumerate}
\end{lemma}

Property~2 is the reason higher-order derivatives do not help.
Outside $H_j$, the derivatives of orders $2,\ldots,p$ all vanish, and inside $H_j$ the first derivative alone already
distinguishes $f_j$.

\begin{proof}[Proof sketch of Theorem~\ref{thm:randomized-derivative-lower-bound}]Let $S:=\min_{r\in\mathcal I}
L_r^{\frac{1}{r+\beta_r}}\Delta^{\frac{r+\beta_r-1}{r+\beta_r}}$.
Rescaling Lemma~\ref{lem:hard-family-properties} gives
$m=\Omega_p(S/\varepsilon)$ instances with disjoint stationary
regions and a common oracle response outside the hidden regions.

Fix a random seed and consider the common-response execution.
For $N<m$, at most $N$ regions are queried.
All untouched instances produce the same output, which succeeds
on at most one of them.
Averaging over the seed therefore gives
\[
\frac12
\le \frac1m\sum_{j=1}^m
\mathbb P\!\left(
|F_j'(\widehat x_A(F_j))|<\varepsilon
\right)
\le \frac{N+1}{m}.
\]
Hence $N\ge m/2-1=\Omega_p(S/\varepsilon)$.
The bound holds for $\mathcal O^{1:p}$ and thus for $\mathcal O^1$.
\end{proof}

\section{Lower Bounds for Deterministic Derivative-Only Algorithms}
\label{sec:4}

We next establish a deterministic lower bound under simultaneous
H\"older smoothness conditions.
We allow each oracle query at $x$ to return
$\mathcal O_f^{1:p}(x):=(f'(x),\ldots,f^{(p)}(x))$.

\begin{theorem}[\bf Deterministic derivative-only lower bound]
\label{thm:deterministic-derivative-jet-lower}
There exists a constant $c_{p,\boldsymbol\beta}>0$,
depending only on $p$ and $\boldsymbol\beta$, such that,
for every $L_1,\ldots,L_p,\Delta>0$ and
$0<\varepsilon\le c_{p,\boldsymbol\beta}
\min_{1\le r\le p}
L_r^{\frac{1}{\nu_r}}\Delta^{\frac{\nu_r-1}{\nu_r}}$,
\[
\mathcal C_{\mathrm{det}}^{(\mathcal I,\boldsymbol\beta)}
(\varepsilon;\mathbf L,\Delta,\mathcal O^{1:p})
\ge c_{p,\boldsymbol\beta}
\min_{1\le r\le p}
L_r^{\frac{1}{\nu_r-1}}\Delta
\varepsilon^{-\frac{\nu_r}{\nu_r-1}}.
\]
\end{theorem}
The bound can be written as
$\min_{1\le r\le p}\bigl(L_r^{\frac{1}{\nu_r}}
\Delta^{\frac{\nu_r-1}{\nu_r}}\varepsilon^{-1}\bigr)^{\frac{\nu_r}{\nu_r-1}}$,
that is, the randomized bound of
Theorem~\ref{thm:randomized-derivative-lower-bound} for the condition of
order $r$, raised to the power $\frac{\nu_r}{\nu_r-1}>1$.
This power decreases toward $1$ as $\nu_r$ grows, so the gap between
deterministic and randomized algorithms narrows with the order of
smoothness.
The hard instances are $C^\infty$, and
Theorem~\ref{thm:randomized-derivative-upper-bound} gives a matching upper
bound that queries $f'$ only.

\subsection{Hard Family}
Fix a nonnegative kernel $\varphi\in C_c^\infty(\mathbb R)$
with $\operatorname{supp}\varphi\subset[-\frac12,\frac12]$
and $\int_{\mathbb R}\varphi=1$.\footnote{For example, take
$\varphi(t):=
\begin{cases}
c_\varphi\exp\!\left(-\frac{1}{1-4t^2}\right),
    & |t|<\frac12,\\
0,  & \text{otherwise}.
\end{cases}$, where
$c_\varphi^{-1}:=\int_{-\frac12}^{\frac12}
\exp(-\frac{1}{1-4u^2})\,du$.}
Let $K_{p,\boldsymbol\beta}\ge1$ be sufficiently large,
depending only on $p$ and $\boldsymbol\beta$, and set
$h:=K_{p,\boldsymbol\beta}\max\limits_{1\le r\le p}
(\frac{\varepsilon}{L_r})^{\frac{1}{\nu_r-1}}$
and $\varphi_h(t):=\frac1h\varphi(\frac th)$.
For each finite $S\subset\mathbb R$ satisfying
$0\in S$ and $4\varepsilon h|S|\le\Delta$, define
$E_S:=\bigcup_{s\in S}[s-h,s+h]$,
$\psi_S:=\mathbf1_{E_S}*\varphi_h$, and
\[
f_S(x):=\Delta-2\varepsilon\int_0^x\psi_S(t)\,dt.
\]
Figure~\ref{fig:derivative-hard-family} illustrates the construction
for $S=\{0,x_1,x_2,x_3,\widehat x\}$.

\begin{figure}[ht]
\centering
\includegraphics[width=0.8\linewidth]{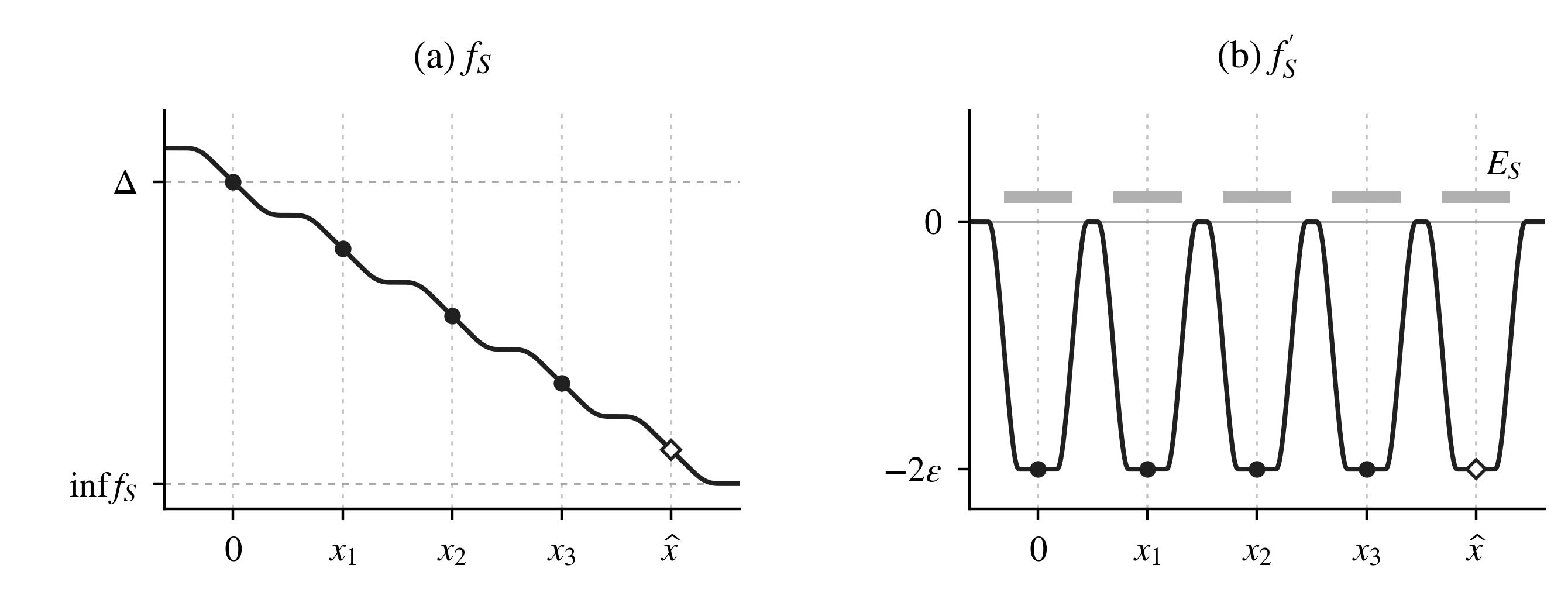}
\vspace{-0.3cm}
\caption{A member of the hard family: (a) $f_S$ and (b) $f_S'$.
Gray bars indicate $E_S$, filled circles mark $0,x_1,x_2,x_3$,
and open diamonds mark $\widehat x$.}
\label{fig:derivative-hard-family}
\vspace{-0.6cm}
\end{figure}
\subsection{Properties and Proof Sketch}
The following lemma records the regularity, initial gap, and oracle values
of the hard family. Its proof is deferred to
Appendix~\ref{app:hard-family-properties}.
\begin{lemma}[\bf Properties of the hard family]
\label{lmm:derivative-hard-family}
For $K_{p,\boldsymbol\beta}$ sufficiently large, every function $f_S$
constructed above satisfies the following properties.
\begin{enumerate}[leftmargin = *,itemsep=2pt, parsep=0pt, topsep=3pt]
    \item \textbf{Regularity and initial gap.}
    The function $f_S\in C^\infty(\mathbb R)$ is nonincreasing and
    bounded below, with $f_S(0)=\Delta$.
    Moreover, $[f_S^{(r)}]_{\beta_r}\le L_r$ for every $1\le r\le p$,
    and $f_S(0)-\inf f_S\le4\varepsilon h|S|\le\Delta$.

    \item \textbf{Oracle values.}
    For every $s\in S$ and $|x-s|<\frac h2$,
    $f_S'(x)=-2\varepsilon$ and
    $f_S^{(j)}(x)=0$ for $2\le j\le p$.
\end{enumerate}
\end{lemma}
Property~2 is the reason higher-order derivatives do not help.
Near every point of $S$, the derivatives of orders $2,\ldots,p$ all vanish,
so a query at such a point reveals only $f_S'=-2\varepsilon$, whatever the
set $S$ is.

\begin{proof}[Proof sketch of Theorem~\ref{thm:deterministic-derivative-jet-lower}]
The accuracy assumption ensures $\varepsilon h\le\frac{\Delta}{16}$.
Fix a deterministic algorithm and answer every query with
$f'=-2\varepsilon$ and $f^{(j)}=0$ for $2\le j\le p$.
If it outputs within $Q$ queries, let $S$ contain the origin,
its query locations, and its output.
For $Q<\frac{\Delta}{16\varepsilon h}$,
we have $4\varepsilon h|S|<\Delta$.
Lemma~\ref{lmm:derivative-hard-family} supplies an admissible $f_S$
that reproduces all simulated replies.
Determinism forces the same execution that the output has derivative
$-2\varepsilon$, or the algorithm exceeds $Q$ queries.
Thus the complexity is at least $\frac{\Delta}{16\varepsilon h}$ and
substituting $h$ gives the claimed bound.
The full proof is deferred to
Appendix~\ref{app:deterministic-derivative-lower}.
\end{proof}

\section{Upper Bounds for Derivative-Only Algorithms}
\label{sec:5}
We next establish upper bounds using only derivative queries.
The algorithm combines piecewise polynomial approximation of $f'$
with importance sampling, using the line-integral identity to control
the success probability.
Applying the construction to the most favorable smoothness condition
among $r\in\mathcal I$ yields the same parameter dependence as
the lower bound.
\begin{theorem}[\bf Derivative-only upper bounds]
\label{thm:randomized-derivative-upper-bound}
Let $\nu_r:=r+\beta_r$ for $r\in\mathcal I$.
There exists a constant $C_{\mathcal I,\boldsymbol\beta}>0$,
depending only on $\mathcal I$ and $\boldsymbol\beta$, such that,
for every
$0<\varepsilon\le
\min_{r\in\mathcal I}
L_r^{\frac{1}{\nu_r}}
\Delta^{\frac{\nu_r-1}{\nu_r}}$,
the following bounds hold:
\begin{align*}
\mathcal C_{\mathrm{det}}^{(\mathcal I,\boldsymbol\beta)}
(\varepsilon;\mathbf L,\Delta,\mathcal O^1)
&\le
C_{\mathcal I,\boldsymbol\beta}
\min_{r\in\mathcal I}
L_r^{\frac{1}{\nu_r-1}}\Delta
\varepsilon^{-\frac{\nu_r}{\nu_r-1}},\\
\mathcal C_{\mathrm{rand}}^{(\mathcal I,\boldsymbol\beta)}
(\varepsilon;\mathbf L,\Delta,\mathcal O^1)
&\le
C_{\mathcal I,\boldsymbol\beta}
\min_{r\in\mathcal I}
L_r^{\frac{1}{\nu_r}}
\Delta^{\frac{\nu_r-1}{\nu_r}}
\varepsilon^{-1}.
\end{align*}
\end{theorem}

The theorem bounds the number of derivative queries required to find
an $\varepsilon$-stationary point, uniformly over the stated function
class. The deterministic algorithm always succeeds, while the
randomized algorithm succeeds with probability at least $1/2$
and satisfies the stated query budget on every execution.
In each bound, the minimum over $r\in\mathcal I$ selects the smallest
expression among the available smoothness conditions.
The constant $C_{\mathcal I,\boldsymbol\beta}$ is independent of
$\varepsilon$, $\mathbf L$, and $\Delta$.

\paragraph{Algorithm.}
By Lemma~\ref{lmm:normal}, it suffices to analyze the normalized
setting $[f^{(p)}]_{\beta_p}\le1$,
$f(0)-\inf f\le1$, and $0<\varepsilon\le1$.
Query the origin and return it if $|f'(0)|<\varepsilon$.
Otherwise, reflect the domain if necessary so that
$f'(0)\le-\varepsilon$, reflecting the final output back afterward.
Set $T:=\frac{2}{\varepsilon}$.

Both deterministic and randomized algorithms use a  deterministic approximation of $f'$.
The following lemma gives the query cost of constructing such an
approximation at a prescribed accuracy.

\begin{lemma}[\bf Approximation]
\label{lmm:derivative-interpolation}
Suppose $[f^{(r)}]_{\beta_r}\le1$ and write $\nu:=r+\beta_r$.
Given an interval $I=[a,b]$ and an accuracy $\delta>0$,
a deterministic algorithm uses
$O_{r,\beta_r}\!\left(
1+|I|\delta^{-\frac{1}{\nu-1}}\right)$
derivative queries to construct a continuous piecewise polynomial
$P_\delta$ satisfying
$\|f'-P_\delta\|_{\infty,I}\le\delta$.
Moreover, $P_\delta(a)=f'(a)$ and $P_\delta(b)=f'(b)$.
\end{lemma}

The lemma provides a deterministic approximation subroutine.
We use accuracy $\delta=\frac{\varepsilon}{8}$ for deterministic
search and $\delta=1$ for randomized search.
The endpoint identities preserve the derivative signs needed
for the level-crossing and bracket arguments.
Its proof is deferred to Appendix~\ref{app:derivative-interpolation}.


We query the origin and return it if it is $\varepsilon$-stationary.
Otherwise, we reflect the function so that $f'(0)\le-\varepsilon$
and apply Algorithm~\ref{alg:derivative-search}.

\noindent
\begin{minipage}[t]{0.56\linewidth}
\vspace{0pt}
\small
\begin{algorithm2e}[H]
\caption{\textsc{DerivativeSearch}}
\label{alg:derivative-search}
\DontPrintSemicolon
\SetKwInOut{Input}{Input}
\Input{$\mathcal O_f^1$, $\varepsilon$, and
$M\in\{\mathrm{det},\mathrm{rand}\}$;
$f'(0)\le-\varepsilon$.}

Set $T:=\frac{2}{\varepsilon}$ and
$\delta:=\frac{\varepsilon}{8}$ if $M=\mathrm{det}$,
otherwise $\delta:=1$.\;
Construct $P$ on $[0,T]$ with error $\delta$ using
Lemma~\ref{lmm:derivative-interpolation}.\;
\If{$M=\mathrm{det}$}{
    \Return any $z\in[0,T]$ with
    $P(z)=-\frac{3\varepsilon}{4}$.\;
}
Set $w:=(P+1+\varepsilon)_+$ and
$\rho:=\frac{w}{\int_0^T w(x)\,dx}$.\;
Set $N:=\left\lceil
\left(1+\frac{4}{\varepsilon}\right)\log2\right\rceil$.\;
\For{$i=1,\ldots,N$}{
    Draw $X\sim\rho$ independently and query $d:=f'(X)$.\;
    \uIf{$|d|<\varepsilon$}{\Return $X$.\;}
    \ElseIf{$d\ge\varepsilon$}{
        \Return \textsc{RefineBracket}$(0,X,\varepsilon)$.\;
    }
}
\Return $0$.\;
\end{algorithm2e}
\end{minipage}\hfill%
\begin{minipage}[t]{0.42\linewidth}
\vspace{0pt}
\small
\begin{algorithm2e}[H]
\caption{\textsc{RefineBracket}}
\label{alg:refine-bracket}
\DontPrintSemicolon
\SetKwInOut{Input}{Input}
\Input{$\mathcal O_f^1$, $a<b$, and $\varepsilon$;
$f'(a)\le-\varepsilon$,
$f'(b)\ge\varepsilon$.}

\While{$b-a>\varepsilon^{\frac{1}{\nu-1}}$}{
    Set $m:=\frac{a+b}{2}$ and query $d:=f'(m)$.\;
    \If{$|d|<\varepsilon$}{\Return $m$.\;}
    \eIf{$d\le-\varepsilon$}
        {$a\gets m$.\;}
        {$b\gets m$.\;}
}
Construct $Q$ on $[a,b]$ with error
$\frac{\varepsilon}{2}$ using
Lemma~\ref{lmm:derivative-interpolation}.\;
\Return any $z\in[a,b]$ with $Q(z)=0$.\;
\end{algorithm2e}
\end{minipage}

In deterministic mode, Algorithm~\ref{alg:derivative-search}
returns a level crossing of the fine approximation.
In randomized mode, it finds a stationary point or a sign bracket
with constant probability using $O_{r,\beta_r}(\varepsilon^{-1})$
queries. In the latter case, Algorithm~\ref{alg:refine-bracket}
uses $O_{r,\beta_r}(1+\log(\varepsilon^{-1}))$ additional queries.
Relative to \textsc{RandomSearch} in~\cite{chewi2023complexity},
we replace uniform sampling by importance sampling and certify
the terminal root using polynomial approximation.

\begin{proof}[Proof sketch of Theorem~\ref{thm:randomized-derivative-upper-bound}]
The initial gap and $T = \frac{2}{\varepsilon}$ imply that $\int_0^T(f'(x)+\varepsilon)\,dx\ge1$.
In deterministic mode, the approximation $P$ must cross
$-\frac{3\varepsilon}{4}$, and its error bound in Lemma~\ref{lmm:derivative-interpolation} certifies stationarity
at every crossing.

In randomized mode, the set $G:=\{x\in[0,T]:f'(x)>-\varepsilon\}$
has weight at least $1$, while its complement has weight at most $2T$.
Thus each sample hits $G$ with probability at least
$\frac{1}{1+4/\varepsilon}$.
The prescribed number of samples succeeds with probability at least
$\frac12$, and bracket refinement certifies a stationary point.
Lemma~\ref{lmm:derivative-interpolation} gives query bounds
$O_{r,\beta_r}(\varepsilon^{-\frac{\nu}{\nu-1}})$ and
$O_{r,\beta_r}(\varepsilon^{-1})$.
Normalization and minimization over $r$ give the theorem.
The proof is deferred to Appendix~\ref{app:derivative-only-upper}.
\end{proof}

\section{Lower Bounds for Randomized Algorithms with Full Information}
\label{sec:6}
We now allow each query at $x$ to return
$\mathcal O_f^{0:p}(x):=(f(x),f'(x),\ldots,f^{(p)}(x))$.
The following theorem gives a logarithmic lower bound for randomized
algorithms under the same simultaneous H\"older conditions.

\begin{theorem}[\bf Randomized lower bound]
\label{thm:randomized-full-jet-lower-bound}
Let $\mathcal I:=\{1,\ldots,p\}$ and $\nu_r:=r+\beta_r$.
There exists a constant $c_{p,\boldsymbol\beta}\in(0,1)$,
depending only on $p$ and $\boldsymbol\beta$, such that,
for every $L_1,\ldots,L_p,\Delta>0$ and
$0<\varepsilon\le c_{p,\boldsymbol\beta}
\min_{1\le r\le p}
L_r^{\frac1{\nu_r}}\Delta^{\frac{\nu_r-1}{\nu_r}}$,
\[
\mathcal C_{\mathrm{rand}}^{(\mathcal I,\boldsymbol\beta)}
(\varepsilon;\mathbf L,\Delta,\mathcal O^{0:p})
\ge c_{p,\boldsymbol\beta}
\log\big(
{
\displaystyle\min_{1\le r\le p}
L_r^{\frac1{\nu_r}}\Delta^{\frac{\nu_r-1}{\nu_r}}
}{\varepsilon^{-1}}
\big).
\]
\end{theorem}
Compared with Theorem~\ref{thm:randomized-derivative-lower-bound},
the dependence on $\varepsilon$ drops from $\varepsilon^{-1}$ to
$\log(1/\varepsilon)$.
The bound holds for randomized algorithms and for the full-information
oracle.
Hence, once function values are available, neither randomness nor
derivatives of any order up to $p$ reduce the query complexity further.
Theorem~\ref{thm:deterministic-zeroth-order-upper} gives a matching upper
bound that is deterministic and queries $f$ only.
\subsection{Hard family}
\label{subsec:randomized-full-jet-hard-family}

We first construct the family for $L_1=\cdots=L_p=\Delta=1$.
Using the fixed kernel $\varphi$ from the preceding section, define $\Phi(t):=\int_{-\infty}^{t-\frac12}\varphi(u)\,du,$ then $\Phi(t)=
\begin{cases}
0, & t\le0,\\
1, & t\ge1.
\end{cases}$
Choose a sufficiently large constant $K_{p,\boldsymbol\beta}\ge1$.
For $0<\varepsilon\le\frac1{4K_{p,\boldsymbol\beta}}$, set
$m:=\left\lfloor
\frac1{2K_{p,\boldsymbol\beta}\varepsilon}
\right\rfloor,
\delta:=\frac1m.$
For each $j\in[m]$, define
\[
f_j(x):=
\frac{1-\delta x+\Phi(x-j+1)}{K_{p,\boldsymbol\beta}},
\qquad x\in[0,m].
\]
Extend $f_j$ periodically with period $m$ to $x\ge0$, and set
$f_j(x):=\frac{1-\delta x}{K_{p,\boldsymbol\beta}}$ for $x<0$.
The hidden interval is $I_j:=[j-1,j]$, with periodic copies
$I_j+km$, $k=0,1,\ldots$. Figure~\ref{fig:randomized-hard-family} illustrates a member of this family
and its derivative.

\begin{figure}[ht]
    \centering
    \includegraphics[width=0.8\textwidth]{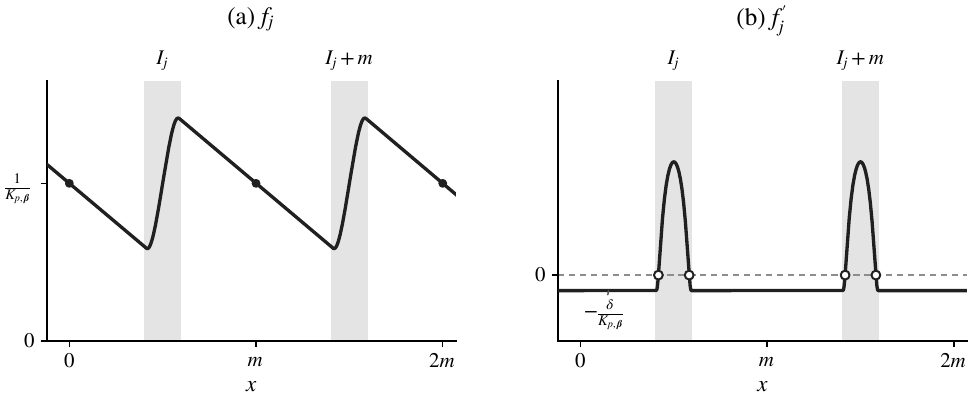}
    \vspace{-0.6cm}
    \caption{The periodic hard family with $m=5$ and $j=3$:
    (a) $f_j$ and (b) $f_j'$.
    Shading marks $I_j$ and $I_j+m$; open circles in (b)
    mark zeros of $f_j'$.}
    \label{fig:randomized-hard-family}
    \vspace{-0.6cm}
\end{figure}

\subsection{Properties and Proof Sketch}

The following lemma records the properties used in the lower bound;
its proof is deferred to Appendix~\ref{app:randomized-hard-family-properties}.

\begin{lemma}[\bf Properties of the hard family]
\label{lmm:randomized-hard-family}
For a sufficiently large $K_{p,\boldsymbol\beta}$, depending only on
$p$ and $\boldsymbol\beta$, the family above has the following properties.
\begin{enumerate}[leftmargin=*, itemsep=2pt, parsep=0pt]
    \item \textbf{Regularity and initial gap.}
    Every $f_j$ is $C^\infty$ and bounded below, with
    $[f_j^{(r)}]_{\beta_r}\le1$ for $1\le r\le p$ and
    $f_j(0)-\inf f_j\le\frac1{K_{p,\boldsymbol\beta}}\le1$.
    Moreover, $f_j'(0)=-\frac{\delta}{K_{p,\boldsymbol\beta}}
    \le-2\varepsilon$.

    \item \textbf{Indistinguishability.}
    For every fixed $x\in\mathbb R$, the response
    $\mathcal O_{f_j}^{0:p}(x)$ takes at most three distinct values
    as $j\in[m]$ varies.

    \item \textbf{Localization of stationary points.}
    For every $j\in[m]$,
    \[
    |f_j'(x)|<\varepsilon
    \quad\Longrightarrow\quad
    x\in\bigcup_{k=0}^{\infty}(j-1+km,j+km).
    \]
        These sets are pairwise disjoint for distinct $j$, so no point is
    $\varepsilon$-stationary for two different indices.
\end{enumerate}
\end{lemma}
Property~2 is the reason function values help, and only logarithmically.
The value $f_j(x)$ reveals whether the hidden interval lies to the left of
$x$, to the right of $x$, or contains $x$, so a query eliminates a constant
fraction of the indices rather than a single one.
The derivatives add nothing to this, since away from the hidden interval
they are the same for every $j$.

\begin{proof}[Proof sketch of Theorem~\ref{thm:randomized-full-jet-lower-bound}]
In the normalized setting $L_1=\cdots=L_p=\Delta=1$,
Lemma~\ref{lmm:randomized-hard-family} gives
$m=\Theta_{p,\boldsymbol\beta}(1/\varepsilon)$ admissible instances
with disjoint stationary sets and at most three replies per query.

Fix the entire random seed of an algorithm using at most $N$ queries.
Its decision tree has at most $3^N$ output leaves.
By localization, each leaf succeeds on at most one instance.
Averaging over the seed therefore gives
\[
\frac12
\le\frac1m\sum_{j=1}^m
\mathbb P\{\text{success on }f_j\}
\le\frac{3^N}{m}.
\]
Hence $N\ge\frac{\log(m/2)}{\log3}
=\Omega_{p,\boldsymbol\beta}(\log(1/\varepsilon))$.
For general parameters, apply Lemma~\ref{lmm:normal} at an index minimizing
$L_r^{\frac1{\nu_r}}\Delta^{\frac{\nu_r-1}{\nu_r}}$.
All normalized H\"older bounds are then at least $1$, so the same hard family
is admissible. The full proof is deferred to
Appendix~\ref{app:randomized-full-jet-lower-bound}.
\end{proof}

\section{Upper Bounds For Deterministic Zeroth-order Algorithms}
\label{sec:7}

We next establish a deterministic upper bound using only function values.
We adapt the bisection approach of~\cite{chewi2023complexity},
using local polynomial interpolation to approximate derivatives.
Applying the algorithm under the most favorable smoothness condition
gives the following bound.

\begin{theorem}[\bf Deterministic function-value upper bound]
\label{thm:deterministic-zeroth-order-upper}
There exists a constant $C_{\mathcal I,\boldsymbol\beta}>0$, depending only
on $\mathcal I$ and $\boldsymbol\beta$, such that, for every
$\mathbf L\in(0,\infty)^{\mathcal I}$, $\Delta>0$, and $\varepsilon>0$,
\[
\mathcal C_{\mathrm{det}}^{(\mathcal I,\boldsymbol\beta)}
(\varepsilon;\mathbf L,\Delta,\mathcal O^0)
\le
C_{\mathcal I,\boldsymbol\beta}
\big(
1+\log_+\big(
\min_{r\in\mathcal I}
L_r^{\frac{1}{\nu_r}}
\Delta^{\frac{\nu_r-1}{\nu_r}}
\varepsilon^{-1}
\big)
\big),
\]
where $\log_+(t):=\max\{\log t,0\}$.
\end{theorem}
The $\log_+$ appears because the theorem places no upper limit on
$\varepsilon$: when $\varepsilon\ge\min_{r\in\mathcal I}
L_r^{\frac{1}{\nu_r}}\Delta^{\frac{\nu_r-1}{\nu_r}}$, a constant number of
queries suffices.
Together with Theorem~\ref{thm:randomized-full-jet-lower-bound}, the bound
shows that $\Theta(\log(1/\varepsilon))$ queries are necessary and
sufficient with function values, for both deterministic and randomized
algorithms.

\paragraph{Algorithm}
Fix $r\in\mathcal I$ and write $\nu:=r+\beta_r$.
By Lemma~\ref{lmm:normal}, we work with
$[f^{(r)}]_{\beta_r}\le1$, $f(0)-\inf f\le1$, and
$0<\varepsilon\le1$. We first record a function-value approximation lemma.
Its proof is deferred to Appendix~\ref{app:value-interpolation}.

\begin{lemma}[\bf Function-value approximation]
\label{lmm:value-interpolation}
There exists a constant $a_r\ge1$, depending only on $r$, such that,
given an interval $I=[a,b]$ with $a<b$, a deterministic algorithm
uses $O_r(1)$ function-value queries to construct a polynomial
$\mathsf Q_I$ satisfying
\[
\|f'-\mathsf Q_I'\|_{\infty,I}\le a_r|I|^{\nu-1},
\qquad
\mathsf Q_I(a)=f(a),\quad \mathsf Q_I(b)=f(b).
\]
\end{lemma}

Set $h_\varepsilon:=
\big(\frac{\varepsilon}{8a_r}\big)^{\frac{1}{\nu-1}}$.
Using Lemma~\ref{lmm:value-interpolation} on an interval $I$
of length $h_\varepsilon$ with endpoint $x$,
compute $D(x):=\mathsf Q_I'(x)$.
This records $f(x)$ and guarantees
$|D(x)-f'(x)|\le\frac{\varepsilon}{8}$.
\begin{algorithm2e}[ht]
\caption{\textsc{SecantSearch}}
\label{alg:secant-search}
\DontPrintSemicolon
\SetKwInOut{Input}{Input}
\Input{$\mathcal O_f^0$, $0<\varepsilon\le1$, and cached
$D(0)<-\frac{3\varepsilon}{4}$.}
Set $a:=0$, $b:=\frac{4}{\varepsilon}$; compute $D(b)$.\;
\lIf{$|D(b)|\le\frac{3\varepsilon}{4}$}{\Return $b$}
\While{$b-a>h_\varepsilon$}{
    Set $c:=\frac{a+b}{2}$; compute $D(c)$.\;
    \lIf{$|D(c)|\le\frac{3\varepsilon}{4}$}{\Return $c$}
    \uIf{$D(c)>\frac{3\varepsilon}{4}$}{Set $b:=c$.\;}
    \uElseIf{$D(b)>\frac{3\varepsilon}{4}$}{Set $a:=c$.\;}
    \Else{$[a,b]\leftarrow
    \begin{cases}
        [a,c], & 2f(c)\ge f(a)+f(b),\\
        [c,b], & 2f(c)<f(a)+f(b).
    \end{cases}$\;\;}
}
Construct $\mathsf Q_{[a,b]}$.\;
\Return any $z\in[a,b]$ with
$\mathsf Q_{[a,b]}'(z)=-\frac{\varepsilon}{2}$.\;
\end{algorithm2e}

Compute $D(0)$ and return $0$ if
$|D(0)|\le\frac{3\varepsilon}{4}$.
Otherwise, reflect if necessary, negating the cached estimate so that
$D(0)<-\frac{3\varepsilon}{4}$.
Apply Algorithm~\ref{alg:secant-search}, undoing any reflection on its output.

A small derivative estimate certifies stationarity.
Otherwise, the search maintains a negative left-endpoint derivative
and either a sign bracket or a secant slope at least
$-\frac{\varepsilon}{4}$.
The first two updates preserve a sign bracket; the remaining updates
select the half with the larger secant slope.
These certificates guarantee the terminal polynomial crossing,
where the approximation error gives
$|f'(z)|\le\frac{5\varepsilon}{8}<\varepsilon$.

\begin{proof}[Proof sketch of Theorem~\ref{thm:deterministic-zeroth-order-upper}]
The initial gap gives
$\frac{f(4/\varepsilon)-f(0)}{4/\varepsilon}
\ge-\frac{\varepsilon}{4}$.
Each iteration preserves $D(a)<-\frac{3\varepsilon}{4}$ and either
$D(b)>\frac{3\varepsilon}{4}$ or this secant lower bound.
The estimate error certifies the endpoint signs, and the parent
secant is the average of the two half-interval secants.
Thus the appropriate half retains a certificate while its length halves.

At length at most $h_\varepsilon$, the value interpolant satisfies
$\|\mathsf Q_{[a,b]}'-f'\|_\infty\le\frac{\varepsilon}{8}$.
Its derivative is at most $-\frac{\varepsilon}{2}$ at $a$
and either at least $\frac{\varepsilon}{2}$ at $b$ or has average
at least $-\frac{\varepsilon}{4}$.
Hence it crosses $-\frac{\varepsilon}{2}$, yielding
$|f'(z)|\le\frac{5\varepsilon}{8}<\varepsilon$.
Each iteration uses $O_r(1)$ function-value queries, and the number
of iterations is $O_{r,\beta_r}(1+\log(1/\varepsilon))$.
Normalization, with accuracy capped at $1$, and minimization over $r$
give the theorem.
The full proof is deferred to Appendix~\ref{app:function-value-upper}.
\end{proof}

\section{Conclusion}
We determined the query complexity of finding approximate stationary points
of one-dimensional functions under higher-order H\"older smoothness, for
deterministic and randomized algorithms and for each of the oracles
considered.
The order of smoothness affects only the deterministic derivative-only
complexity, and higher-order derivative queries improve none of the bounds.
By embedding our hard instances, we also showed that the
$\varepsilon$-dependence of the H\"older extension of
OFFAR$p$~\citep{gratton2023convergence} is optimal in every dimension, and
that of deterministic gradient-only trapping~\citep{ZSS26} in dimension two.

Several questions remain open.
First, the complexity of randomized gradient-only algorithms in dimension
two is not settled.
The trapping algorithm of \cite{ZSS26} makes
$\widetilde O(\varepsilon^{-\frac{2\nu_p}{2\nu_p-1}})$ queries, whereas the
only lower bound we know is the $\Omega(\varepsilon^{-1})$ obtained by
embedding.
Second, higher-order derivative queries do not help in one dimension but do
help in sufficiently high dimension, and randomness shows the opposite
behavior.
The dimensions at which these transitions occur are not known.

\section*{AI Disclosure}

AI tools were used to assist with literature searches, the development and verification of proof
arguments, the presentation of proofs, the polishing of the exposition, and typesetting. The authors
independently verified all mathematical arguments and take full responsibility for the content of the
paper.

\section*{Acknowledgment}
We thank Taoli Zheng for helpful discussions. TS was partially supported by
JST CREST (JPMJCR2115) and JSPS KAKENHI (24K02905). 
This work was partially supported by JST ERATO (JPMJER2601). 
This research is supported by the
National Research Foundation, Singapore and the Ministry of Digital Development and Information
under the AI Visiting Professorship Programme (award number AIVP-2024-004). Any opinions,
findings and conclusions or recommendations expressed in this material are those of the author(s)
and do not reflect the views of National Research Foundation, Singapore and the Ministry of Digital
Development and Information.

\newpage
\bibliography{reference}

\newcommand{\etalchar}[1]{$^{#1}$}
\begin{thebibliography}{AAZB{\etalchar{+}}17}

\bibitem[AAZB{\etalchar{+}}17]{agarwal2017finding}
Naman Agarwal, Zeyuan Allen-Zhu, Brian Bullins, Elad Hazan, and Tengyu Ma.
\newblock Finding approximate local minima faster than gradient descent.
\newblock In {\em Proceedings of the 49th annual ACM SIGACT symposium on theory
  of computing}, pages 1195--1199, 2017.

\bibitem[BCN18]{BCN18}
L{\'{e}}on Bottou, Frank~E. Curtis, and Jorge Nocedal.
\newblock {Optimization Methods for Large-Scale Machine Learning}.
\newblock {\em {SIAM} Rev.}, 60(2):223--311, 2018.

\bibitem[BGM{\etalchar{+}}17]{birgin2017worst}
Ernesto~G Birgin, JL~Gardenghi, Jos{\'e}~Mario Mart{\'\i}nez, Sandra~Augusta
  Santos, and Ph~L Toint.
\newblock Worst-case evaluation complexity for unconstrained nonlinear
  optimization using high-order regularized models.
\newblock {\em Mathematical Programming}, 163:359--368, 2017.

\bibitem[BM20]{bubeck2020trap}
S{\'e}bastien Bubeck and Dan Mikulincer.
\newblock How to trap a gradient flow.
\newblock In {\em Proceedings of the Thirty-Third Conference on Learning
  Theory}, volume 125 of {\em Proceedings of Machine Learning Research}, pages
  940--960. PMLR, 2020.

\bibitem[CBS23]{chewi2023complexity}
Sinho Chewi, S{\'e}bastien Bubeck, and Adil Salim.
\newblock On the complexity of finding stationary points of smooth functions in
  one dimension.
\newblock In {\em International Conference on Algorithmic Learning Theory},
  pages 358--374. PMLR, 2023.

\bibitem[CDHS17a]{carmon2017convex}
Yair Carmon, John~C. Duchi, Oliver Hinder, and Aaron Sidford.
\newblock {``Convex Until Proven Guilty'': Dimension-Free Acceleration of
  Gradient Descent on Non-Convex Functions}.
\newblock In Doina Precup and Yee~Whye Teh, editors, {\em Proceedings of the
  34th International Conference on Machine Learning}, volume~70 of {\em
  Proceedings of Machine Learning Research}, pages 654--663. PMLR, 2017.

\bibitem[CDHS17b]{carmon2017lower}
Yair Carmon, John~C Duchi, Oliver Hinder, and Aaron Sidford.
\newblock Lower bounds for finding stationary points i.
\newblock {\em arXiv preprint arXiv:1710.11606}, 2017.

\bibitem[CDHS18]{carmon2018accelerated}
Yair Carmon, John~C Duchi, Oliver Hinder, and Aaron Sidford.
\newblock Accelerated methods for nonconvex optimization.
\newblock {\em SIAM Journal on Optimization}, 28(2):1751--1772, 2018.

\bibitem[CGT11]{cartis2011adaptive}
Coralia Cartis, Nicholas I.~M. Gould, and Philippe~L. Toint.
\newblock Adaptive cubic regularisation methods for unconstrained optimization.
  {Part I}: Motivation, convergence and numerical results.
\newblock {\em Mathematical Programming}, 127(2):245--295, 2011.

\bibitem[GJT21]{gratton2021h}
Serge Gratton, Sadok Jerad, and Philippe~L Toint.
\newblock H$\backslash$" older gradient descent and adaptive regularization
  methods in banach spaces for first-order points.
\newblock {\em arXiv preprint arXiv:2104.02564}, 2021.

\bibitem[GJT23]{gratton2023convergence}
Serge Gratton, Sadok Jerad, and Philippe~L. Toint.
\newblock Convergence properties of an objective-function-free optimization
  regularization algorithm, including an {$\mathcal{O}(\epsilon^{-3/2})$}
  complexity bound.
\newblock {\em SIAM Journal on Optimization}, 33(3):1621--1646, 2023.

\bibitem[HZ23]{hollender2023computational}
Alexandros Hollender and Manolis Zampetakis.
\newblock The computational complexity of finding stationary points in
  non-convex optimization.
\newblock In {\em Proceedings of the Thirty-Sixth Conference on Learning
  Theory}, volume 195 of {\em Proceedings of Machine Learning Research}, pages
  5571--5572. PMLR, 2023.

\bibitem[LL23]{li2023restarted}
Huan Li and Zhouchen Lin.
\newblock Restarted nonconvex accelerated gradient descent: No more
  polylogarithmic factor in the in the o (epsilon\^{}(-7/4)) complexity.
\newblock {\em Journal of Machine Learning Research}, 24(157):1--37, 2023.

\bibitem[LW15]{LW15}
Po{-}Ling Loh and Martin~J. Wainwright.
\newblock {Regularized M-estimators with nonconvexity: statistical and
  algorithmic theory for local optima}.
\newblock {\em J. Mach. Learn. Res.}, 16:559--616, 2015.

\bibitem[MT24]{marumo2024parameter}
Naoki Marumo and Akiko Takeda.
\newblock Parameter-free accelerated gradient descent for nonconvex
  minimization.
\newblock {\em SIAM Journal on Optimization}, 34(2):2093--2120, 2024.

\bibitem[NP06]{nesterov2006cubic}
Yurii Nesterov and Boris~T. Polyak.
\newblock Cubic regularization of {Newton} method and its global performance.
\newblock {\em Mathematical Programming}, 108(1):177--205, 2006.

\bibitem[NW99]{NW06}
Jorge Nocedal and Stephen~J. Wright.
\newblock {\em {Numerical Optimization}}.
\newblock Springer, 1999.

\bibitem[PV91]{PV91}
Panos~M. Pardalos and Stephen~A. Vavasis.
\newblock Quadratic programming with one negative eigenvalue is {NP}-hard.
\newblock {\em Journal of Global Optimization}, 1(1):15--22, 1991.

\bibitem[Vav93]{vavasis1993black}
Stephen~A Vavasis.
\newblock Black-box complexity of local minimization.
\newblock {\em SIAM Journal on Optimization}, 3(1):60--80, 1993.

\bibitem[ZHTS25]{zhou2025adaptive}
Huanjian Zhou, Andi Han, Akiko Takeda, and Masashi Sugiyama.
\newblock The adaptive complexity of finding a stationary point.
\newblock In Nika Haghtalab and Ankur Moitra, editors, {\em Proceedings of the
  Thirty-Eighth Conference on Learning Theory}, volume 291 of {\em Proceedings
  of Machine Learning Research}, pages 6091--6123. PMLR, 2025.

\bibitem[ZSS26]{ZSS26}
Huanjian Zhou, Masashi Sugiyama, and Taiji Suzuki.
\newblock Finding stationary points under higher-order smoothness: Trapping
  algorithms and first-order lower bounds.
\newblock {\em arXiv preprint arXiv:2610.02791}, 2026.

\end{thebibliography}
\bibliographystyle{alpha}

\newpage

\appendix

\section{Proof of Lemma~\ref{lmm:normal}}
\label{app:1}
\begin{proof}
Define $g(x):=f(sx)/\Delta$. For every $r\in\mathcal I$,
\[
[g^{(r)}]_{\beta_r}
=
\frac{s^{r+\beta_r}}{\Delta}[f^{(r)}]_{\beta_r},
\qquad
g(0)-\inf g
=
\frac{f(0)-\inf f}{\Delta}.
\]
Consequently,
\[
f\in\mathcal F_{\mathcal I,\boldsymbol\beta}(\mathbf L,\Delta)
\quad\Longleftrightarrow\quad
g\in\mathcal F_{\mathcal I,\boldsymbol\beta}
(\widehat{\mathbf L},1),
\]
with inverse $f(y)=\Delta g(y/s)$.

For $0\le r\le p$, with $f^{(0)}:=f$ and $g^{(0)}:=g$,
\[
g^{(r)}(x)=\frac{s^r}{\Delta}f^{(r)}(sx),
\qquad
f^{(r)}(y)=\frac{\Delta}{s^r}g^{(r)}(y/s).
\]
Thus each oracle query can be simulated by one query in either
direction. Moreover,
\[
|g'(x)|<\frac{s\varepsilon}{\Delta}
\quad\Longleftrightarrow\quad
|f'(sx)|<\varepsilon,
\qquad
\frac{s\varepsilon}{\Delta}
=
\frac{\varepsilon}
{L_p^{1/\nu_p}\Delta^{(\nu_p-1)/\nu_p}}.
\]
The two reductions preserve query counts and success probabilities,
proving the claimed equality.
\end{proof}

\section{Proofs in Section~\ref{sec:randomized-lower-bound}}
\subsection{Proof of Lemma~\ref{lem:hard-family-properties}}
\label{app:2}

\begin{proof}[Proof of Lemma~\ref{lem:hard-family-properties}]
Since $(2K\varepsilon)^{-1}\ge4$ and
$\lfloor t\rfloor\ge t/2$ for $t\ge2$,
\[
m\ge\max\left\{4,\frac{1}{4K\varepsilon}\right\},
\qquad
2K\varepsilon\le\delta=\frac1m\le\frac14.
\]

\paragraph{Regularity and initial gap.}
The endpoint identities for $\eta_p$ and the equality $\delta m=1$
ensure matching values and derivatives through order $p$ at every
junction, including the origin. Thus $f_j\in C^p(\mathbb R)$.
On $[0,m]$, we have $1-\delta x\ge0$ and $\eta_p\ge0$;
on $(-\infty,0]$, we have $Kf_j(x)=1-\delta x\ge1$.
Periodicity therefore gives
\[
f_j\ge0,
\qquad
f_j(0)-\inf f_j\le f_j(0)=\frac1K\le1.
\]

Fix $r\in\mathcal I$.
The function
$f_j^{(r)}+\frac{\delta}{K}\mathbf1_{\{r=1\}}$
is zero outside $H_j$ and a translate of $\eta_p^{(r)}/K$
on each hidden interval. Hence
\[
\left\|
f_j^{(r)}+\frac{\delta}{K}\mathbf1_{\{r=1\}}
\right\|_\infty
\le\frac{\|\eta_p^{(r)}\|_{\infty,[0,1]}}{K}
\le\frac12.
\]
Each piece is Lipschitz with constant at most
$\|\eta_p^{(r+1)}\|_{\infty,[0,1]}/K\le1/2$.
Continuity allows these bounds to be summed across the finitely
many junctions between any $x,y\in\mathbb R$. Consequently,
\[
|f_j^{(r)}(x)-f_j^{(r)}(y)|
\le\min\left\{\frac{|x-y|}{2},1\right\}
\le |x-y|^{\beta_r}.
\]
The last inequality follows by considering $|x-y|\le1$ and
$|x-y|>1$ separately.
Thus $[f_j^{(r)}]_{\beta_r}\le1$ for every $r\in\mathcal I$,
proving $f_j\in\mathcal F_{\mathcal I,\boldsymbol\beta}(\mathbf1,1)$.

\paragraph{Indistinguishability.}
The construction and the endpoint identities give
\[
x\notin H_j
\quad\Longrightarrow\quad
f_j'(x)=-\frac{\delta}{K},
\qquad
f_j^{(r)}(x)=0\quad(2\le r\le p).
\]
Hence $\mathcal O_{f_j}^{1:p}(x)=(-\delta/K,0,\ldots,0)$
outside $H_j$, with $\delta/K\ge2\varepsilon$.

\paragraph{Localization of stationary points.}
The preceding identity and the definition of $H_j$ imply
\[
\{x:|f_j'(x)|<\varepsilon\}\subseteq H_j,
\qquad
H_i\cap H_j=\varnothing\quad(i\ne j).
\]
Moreover,
\[
f_j'(j-1)=f_j'(j)=-\frac{\delta}{K}<0,
\qquad
\int_{j-1}^{j}f_j'(x)\,dx=\frac{1-\delta}{K}>0.
\]
Thus $f_j'$ is positive somewhere in $(j-1,j)$.
Continuity gives a zero on either side of that point, and
periodicity yields stationary points in every component of $H_j$.
\end{proof}

\subsection{Proof of Theorem~\ref{thm:randomized-derivative-lower-bound}}

\begin{proof}[Proof of Theorem~\ref{thm:randomized-derivative-lower-bound}]
Set $c_p:=(16K)^{-1}$ and
\[
b:=
\max_{r\in\mathcal I}
\left(\frac{\Delta}{L_r}\right)^{\frac{1}{r+\beta_r}},
\qquad
S=\frac{\Delta}{b},
\qquad
\widehat\varepsilon:=\frac{\varepsilon}{S}.
\]
Since $\varepsilon\le c_pS$, we have
$\widehat\varepsilon\le(8K)^{-1}$.
Apply Lemma~\ref{lem:hard-family-properties} with accuracy
$\widehat\varepsilon$, and define $F_j(x):=\Delta f_j(x/b)$.

\paragraph{Rescaling.}
For $1\le r\le p$, the chain rule gives
$F_j^{(r)}(x)=\Delta b^{-r}f_j^{(r)}(x/b)$.
Consequently, for every $r\in\mathcal I$,
\[
[F_j^{(r)}]_{\beta_r}
=
\frac{\Delta}{b^{r+\beta_r}}[f_j^{(r)}]_{\beta_r}
\le
\frac{\Delta}{b^{r+\beta_r}}
\le L_r,
\qquad
F_j(0)-\inf F_j
=
\Delta\bigl(f_j(0)-\inf f_j\bigr)
\le\Delta.
\]
Thus $F_j\in\mathcal F_{\mathcal I,\boldsymbol\beta}
(\mathbf L,\Delta)$.
Since $F_j'(x)=Sf_j'(x/b)$, a point $x$ is
$\varepsilon$-stationary for $F_j$ if and only if $x/b$ is
$\widehat\varepsilon$-stationary for $f_j$.
Every successful output therefore belongs to $bH_j$, and these
sets are pairwise disjoint.
Furthermore, outside $bH_j$, the oracle returns the common response
\[
\mathcal O_{F_j}^{1:p}(x)
=
v:=\left(-\frac{S\delta}{K},0,\ldots,0\right).
\]

\paragraph{Query counting.}
Let $A$ use at most $N$ queries on every run.
The case $N\ge m$ is immediate, so assume $N<m$.
Fix a random seed $r$ and run $A$ with the constant response $v$.
This execution terminates within $N$ queries, since its first
$N$ responses remain consistent with at least $m-N>0$ instances.

Let $Q_r$ be the set of indices whose regions $bH_j$ are queried,
and let $\bar x_r$ be the output.
Then $|Q_r|\le N$, and indistinguishability gives
$\widehat x_A(F_j;r)=\bar x_r$ for every $j\notin Q_r$.
By localization and disjointness,
\[
\sum_{j=1}^m
\mathbf 1\!\left\{
    |F_j'(\widehat x_A(F_j;r))|<\varepsilon
\right\}
\le
|Q_r|+
\sum_{j\notin Q_r}\mathbf 1\{\bar x_r\in bH_j\}
\le N+1.
\]
Averaging over the random seed and using the success guarantee,
\[
\frac12
\le
\frac1m\sum_{j=1}^m
\mathbb P\!\left(
    |F_j'(\widehat x_A(F_j))|<\varepsilon
\right)
\le\frac{N+1}{m}.
\]
Since $m\ge4$ and $m\ge(4K\widehat\varepsilon)^{-1}$,
\[
N\ge\frac m2-1
\ge\frac m4
\ge\frac{1}{16K\widehat\varepsilon}
=c_p\frac{S}{\varepsilon}.
\]
The bound holds for $\mathcal O^{1:p}$ and therefore also
for $\mathcal O^1$.
\end{proof}

\section{Proofs in Section~\ref{sec:4}}
\subsection{Proof of Lemma~\ref{lmm:derivative-hard-family}}
\label{app:hard-family-properties}

\begin{proof}
Since $\varphi$ is fixed, choose
\[
M_p:=\max_{0\le k\le p}
\|\varphi^{(k)}\|_{L^1(\mathbb R)}\in[1,\infty),
\qquad
K_{p,\boldsymbol\beta}\ge
\max_{1\le r\le p}(4M_p)^{\frac1{\nu_r-1}}.
\]

\noindent\textbf{Regularity and initial gap.}
Smoothness and compact support of the kernel give
\[
\begin{aligned}
&\psi_S\in C_c^\infty(\mathbb R),\qquad
  \operatorname{supp}\psi_S\subset E_S+[-\frac h2,\frac h2],\\
&0\le\psi_S(x)
  =\int_{\mathbb R}\mathbf1_{E_S}(x-t)\varphi_h(t)\,dt
  \le\int_{\mathbb R}\varphi_h(t)\,dt=1,\\
&f_S\in C^\infty(\mathbb R),\qquad f_S(0)=\Delta,\qquad
  f_S'=-2\varepsilon\psi_S\le0.
\end{aligned}
\]
Thus $f_S$ is nonincreasing, constant on both tails, and bounded below.
For $0\le k\le p$, differentiation and scaling yield
\[
\begin{aligned}
\varphi_h^{(k)}(t)
  &=\frac1{h^{k+1}}\varphi^{(k)}(\frac th),
&\|\varphi_h^{(k)}\|_1
  &=h^{-k}\|\varphi^{(k)}\|_1,\\
\psi_S^{(k)}&=\mathbf1_{E_S}*\varphi_h^{(k)},
&\|\psi_S^{(k)}\|_\infty
  &\le\|\mathbf1_{E_S}\|_\infty\|\varphi_h^{(k)}\|_1
   \le M_ph^{-k}.
\end{aligned}
\]
Consequently, for $1\le r\le p$,
\[
\begin{aligned}
f_S^{(r)}&=-2\varepsilon\psi_S^{(r-1)},
&f_S^{(r+1)}&=-2\varepsilon\psi_S^{(r)},\\
\|f_S^{(r)}\|_\infty&\le2\varepsilon M_ph^{-(r-1)},
&\|f_S^{(r+1)}\|_\infty&\le2\varepsilon M_ph^{-r}.
\end{aligned}
\]
For $d:=|x-y|>0$, the triangle inequality
and the mean value theorem give
\[
\begin{aligned}
|f_S^{(r)}(x)-f_S^{(r)}(y)|
&\le\min\{2\|f_S^{(r)}\|_\infty,
              \|f_S^{(r+1)}\|_\infty d\}\\
&\le\min\{4\varepsilon M_ph^{-(r-1)},
              2\varepsilon M_ph^{-r}d\}\\
&\le4\varepsilon M_ph^{-(r-1)}\min\{1,\frac dh\}\\
&\le4\varepsilon M_ph^{-(\nu_r-1)}d^{\beta_r},
\end{aligned}
\]
where $\min\{1,t\}\le t^{\beta_r}$ for $t\ge0$ and
$0<\beta_r\le1$.
Since $h^{\nu_r-1}\ge
K_{p,\boldsymbol\beta}^{\nu_r-1}\frac{\varepsilon}{L_r}$,
\[
[f_S^{(r)}]_{\beta_r}
\le4M_p\varepsilon h^{-(\nu_r-1)}
\le\frac{4M_p}{K_{p,\boldsymbol\beta}^{\nu_r-1}}L_r
\le L_r.
\]
Tonelli's theorem gives
\[
\begin{aligned}
\int_{\mathbb R}\psi_S(x)\,dx
&=\int_{\mathbb R}\varphi_h(t)
  \left(\int_{\mathbb R}\mathbf1_{E_S}(x-t)\,dx\right)dt\\
&=|E_S|\int_{\mathbb R}\varphi_h(t)\,dt
 =|E_S|\le2h|S|.
\end{aligned}
\]
Because the right-tail value of $f_S$ equals its infimum,
\[
\begin{aligned}
f_S(0)-\inf f_S
&=2\varepsilon\int_0^\infty\psi_S(t)\,dt\\
&\le2\varepsilon\int_{\mathbb R}\psi_S(t)\,dt
 \le4\varepsilon h|S|\le\Delta.
\end{aligned}
\]

\noindent\textbf{Oracle values.}
For $s\in S$, $|x-s|<\frac h2$, and
$t\in\operatorname{supp}\varphi_h$,
\[
|x-t-s|\le|x-s|+|t|<\frac h2+\frac h2=h,
\qquad x-t\in[s-h,s+h]\subset E_S.
\]
Consequently, throughout $(s-\frac h2,s+\frac h2)$,
\[
\begin{aligned}
\psi_S(x)
&=\int_{\mathbb R}\mathbf1_{E_S}(x-t)\varphi_h(t)\,dt
 =\int_{\mathbb R}\varphi_h(t)\,dt=1,\\
f_S'(x)&=-2\varepsilon,\qquad
f_S^{(j)}(x)=-2\varepsilon\psi_S^{(j-1)}(x)=0
\quad(2\le j\le p).
\end{aligned}
\]
\end{proof}

\subsection{Proof of Theorem~\ref{thm:deterministic-derivative-jet-lower}}
\label{app:deterministic-derivative-lower}

\begin{proof}
Take the constant in the theorem small enough that
$c_{p,\boldsymbol\beta}\le\frac1{16K_{p,\boldsymbol\beta}}$.
For every $1\le r\le p$, the accuracy assumption gives
$\varepsilon\le c_{p,\boldsymbol\beta}
L_r^{\frac1{\nu_r}}\Delta^{\frac{\nu_r-1}{\nu_r}}$.
Since $\frac{\nu_r}{\nu_r-1}>1$ and
$0<c_{p,\boldsymbol\beta}<1$,
\[
\begin{aligned}
\varepsilon^{\frac{\nu_r}{\nu_r-1}}L_r^{-\frac1{\nu_r-1}}
&\le c_{p,\boldsymbol\beta}^{\frac{\nu_r}{\nu_r-1}}\Delta
 \le c_{p,\boldsymbol\beta}\Delta
 \le\frac{\Delta}{16K_{p,\boldsymbol\beta}},\\
\varepsilon h
&=K_{p,\boldsymbol\beta}\max_{1\le r\le p}
  \varepsilon^{\frac{\nu_r}{\nu_r-1}}L_r^{-\frac1{\nu_r-1}}
 \le\frac{\Delta}{16}.
\end{aligned}
\]

Suppose a correct deterministic algorithm has worst-case query budget
$Q<\frac{\Delta}{16\varepsilon h}$.
Simulate it by reporting $f'(x)=-2\varepsilon$ and
$f^{(j)}(x)=0$ for $2\le j\le p$ at every query.
If it outputs within $Q$ queries, let $S$ contain the origin,
all query locations, and its output $\widehat x$.
If it requests query $Q+1$, let $S$ contain the origin and the
first $Q+1$ query locations.
In either case, $0\in S$, $|S|\le Q+2$, and
\[
4\varepsilon h|S|
\le4\varepsilon h(Q+2)
<\frac{\Delta}{4}+8\varepsilon h
\le\frac{3\Delta}{4}<\Delta.
\]
Lemma~\ref{lmm:derivative-hard-family} therefore supplies an admissible
$f_S\in C^\infty(\mathbb R)$ satisfying
\[
f_S'(s)=-2\varepsilon,\qquad
f_S^{(j)}(s)=0
\quad(s\in S,\ 2\le j\le p).
\]

Matching oracle replies and determinism reproduce the simulation
on $f_S$ inductively.
Thus the algorithm either outputs $\widehat x\in S$ with
$|f_S'(\widehat x)|=2\varepsilon>\varepsilon$, or requests query
$Q+1$, contradicting correctness or the query budget.
Nontermination is likewise impossible: taking $S$ to contain the origin
and the at most $Q$ queried points, the corresponding admissible $f_S$
reproduces the same nonterminating execution.

Hence every correct algorithm needs at least
\[
\frac{\Delta}{16\varepsilon h}
=\frac1{16K_{p,\boldsymbol\beta}}
\min_{1\le r\le p}
L_r^{\frac1{\nu_r-1}}\Delta
\varepsilon^{-\frac{\nu_r}{\nu_r-1}}
\]
queries. Since $c_{p,\boldsymbol\beta}\le\frac1{16K_{p,\boldsymbol\beta}}$,
this proves the stated bound; moreover, $0\in S$ gives
$f_S'(0)=-2\varepsilon$.
\end{proof}




\section{Proofs in Section~\ref{sec:5}}
\subsection{Proof of Theorem~\ref{thm:randomized-derivative-upper-bound}}
\label{app:derivative-only-upper}

\begin{proof}[Proof of Theorem~\ref{thm:randomized-derivative-upper-bound}]
Fix $r\in\mathcal I$ and write $\nu:=r+\beta_r$.
We first assume $[f^{(r)}]_{\beta_r}\le1$,
$f(0)-\inf f\le1$, and $0<\varepsilon\le1$.
The initial query either returns an $\varepsilon$-stationary point
or, after reflection if necessary, gives $f'(0)\le-\varepsilon$.
Reflection preserves the smoothness conditions and initial gap.
Set $T:=\frac{2}{\varepsilon}$.
The fundamental theorem of calculus and the gap bound give
\begin{equation}
\label{eq:derivative-upper-integral}
\int_0^T f'(x)\,dx=f(T)-f(0)\ge-1,
\qquad
\int_0^T(f'(x)+\varepsilon)\,dx\ge1.
\end{equation}

\paragraph{Deterministic algorithm.}
Let $P$ be the approximation with error $\frac{\varepsilon}{8}$.
By Eq.~\eqref{eq:derivative-upper-integral},
\[
\frac1T\int_0^T P(x)\,dx
\ge-\frac1T-\frac{\varepsilon}{8}
=-\frac{5\varepsilon}{8}.
\]
Thus $P$ attains a value strictly above $-\frac{3\varepsilon}{4}$.
Since $P(0)=f'(0)\le-\varepsilon$ and $P$ is continuous,
it crosses this level.
Every such crossing $z$ satisfies
\[
|f'(z)|
\le |P(z)|+\frac{\varepsilon}{8}
=\frac{7\varepsilon}{8}<\varepsilon.
\]
A crossing can be found from the polynomial pieces without further
oracle queries.
Including the initial query, Lemma~\ref{lmm:derivative-interpolation}
bounds the deterministic query count by
\[
O_{r,\beta_r}\!\left(
1+T\left(\frac{\varepsilon}{8}\right)^{-\frac1{\nu-1}}
\right)
=O_{r,\beta_r}\!\left(\varepsilon^{-\frac{\nu}{\nu-1}}\right).
\]

\paragraph{Randomized algorithm.}
Now let $P$ have error $1$, and set
$w:=(P+1+\varepsilon)_+$ and
$G:=\{x\in[0,T]:f'(x)>-\varepsilon\}$.
Since $w\ge(f'+\varepsilon)_+$,
\[
A:=\int_G w(x)\,dx
\ge\int_G(f'(x)+\varepsilon)\,dx
\ge\int_0^T(f'(x)+\varepsilon)\,dx
\ge1.
\]
On $G^c$, the approximation bound gives
$P+1+\varepsilon\le f'+2+\varepsilon\le2$.
Hence $B:=\int_{G^c}w(x)\,dx\le2T$.
In particular, $W:=A+B>0$, and sampling from $\rho:=\frac{w}{W}$
gives
\[
\mathbb P(X\in G)
=\frac{A}{A+B}
\ge\frac1{1+2T}
=\frac1{1+4/\varepsilon}.
\]
The density is determined by the known polynomial pieces;
its integral and samples require no additional oracle queries.
For $N:=\left\lceil(1+\frac4\varepsilon)\log2\right\rceil$
independent samples,
\[
\mathbb P(\text{all $N$ samples miss }G)
\le\left(1-\frac1{1+4/\varepsilon}\right)^N
\le\exp\!\left(-\frac{N}{1+4/\varepsilon}\right)
\le\frac12.
\]
Every hit either satisfies $|f'(X)|<\varepsilon$
or has $f'(X)\ge\varepsilon$, yielding a sign bracket with the origin.

\paragraph{Bracket refinement and query count.}
Bisection either returns an $\varepsilon$-stationary point
or preserves $f'(a)\le-\varepsilon$ and $f'(b)\ge\varepsilon$.
Starting from length at most $T$, it reaches
$b-a\le h_\varepsilon:=\varepsilon^{\frac1{\nu-1}}$
within
\[
\left\lceil\log_2\left(\frac{T}{h_\varepsilon}\right)\right\rceil
=O_{r,\beta_r}(1+\log(\varepsilon^{-1}))
\]
queries.
Apply Lemma~\ref{lmm:derivative-interpolation} on this interval
with error $\frac{\varepsilon}{2}$.
The resulting continuous piecewise polynomial $Q$ satisfies
$Q(a)=f'(a)<0<f'(b)=Q(b)$, so it has a root $z$.
Every root satisfies
$|f'(z)|\le\frac{\varepsilon}{2}<\varepsilon$.
The terminal approximation uses
\[
O_{r,\beta_r}\!\left(
1+(b-a)\left(\frac{\varepsilon}{2}\right)^{-\frac1{\nu-1}}
\right)=O_{r,\beta_r}(1)
\]
queries.
Thus every hit of $G$ leads to a valid output.
The initial query, coarse approximation, at most $N$ samples,
and at most one bracket refinement together use
$O_{r,\beta_r}(\varepsilon^{-1})$ queries on every execution,
since $0<\varepsilon\le1$.
The success probability is at least $\frac12$.

\paragraph{General parameters.}
By Lemma~\ref{lmm:normal}, the normalized accuracy is
\[
\widehat\varepsilon
:=\frac{\varepsilon}
{L_r^{\frac1\nu}\Delta^{\frac{\nu-1}{\nu}}}.
\]
The theorem's assumption ensures $0<\widehat\varepsilon\le1$.
The algorithms use only the selected smoothness condition,
so the other normalized conditions impose no restriction on the argument.
Substituting $\widehat\varepsilon$ into the two normalized query bounds
gives, respectively,
\[
O_{r,\beta_r}\!\left(
L_r^{\frac1{\nu-1}}\Delta
\varepsilon^{-\frac{\nu}{\nu-1}}\right),
\qquad
O_{r,\beta_r}\!\left(
L_r^{\frac1\nu}\Delta^{\frac{\nu-1}{\nu}}\varepsilon^{-1}\right).
\]
Normalization preserves query counts and success probabilities.
Taking the maximum of the constants over $r\in\mathcal I$
and minimizing each bound over $r$ proves the theorem.
\end{proof}

\subsection{Proof of the Lemma~\ref{lmm:derivative-interpolation}}
\label{app:derivative-interpolation}

\begin{proof}[Proof of Lemma~\ref{lmm:derivative-interpolation}]
Write $\beta:=\beta_r$ and $\alpha:=\nu-1=r-1+\beta>0$.
If $a=b$, one derivative query gives the constant polynomial
$P_\delta\equiv f'(a)$, which satisfies all the required properties.
Henceforth assume $a<b$.

For a subinterval $J=[c,d]\subseteq I$, let $R_J$ be the
degree-$(r-1)$ Taylor polynomial of $f'$ at $c$.
If $r=1$, then $R_J\equiv f'(c)$, and Hölder continuity gives
$\|f'-R_J\|_{\infty,J}\le |J|^\beta$.
If $r\ge2$, the integral Taylor formula gives
\[
f'(x)-R_J(x)
=
\frac{1}{(r-2)!}
\int_c^x
(x-t)^{r-2}
\bigl(f^{(r)}(t)-f^{(r)}(c)\bigr)\,dt.
\]
Since $[f^{(r)}]_\beta\le1$, the absolute value of the difference
inside the integral is at most $|J|^\beta$.
Integrating the kernel therefore yields
$|f'(x)-R_J(x)|
\le \frac{|J|^\beta(x-c)^{r-1}}{(r-1)!}
\le \frac{|J|^\alpha}{(r-1)!}$.
Thus, for every $r\ge1$,
$\|f'-R_J\|_{\infty,J}
\le \frac{|J|^\alpha}{(r-1)!}$.

Let $q:=\max\{r,2\}$, and let $\mathsf P_J$ interpolate $f'$
at $q$ equally spaced nodes of $J$, including both endpoints.
Let $\Lambda_q$ be the Lebesgue constant of these nodes.
It depends only on $q$, since affine changes of the interval
preserve the interpolation operator norm.
The interpolation operator has norm $\Lambda_q$ and reproduces
$R_J$, whose degree is at most $q-1$.
Applying it to $f'-R_J$ and using the triangle inequality gives
\[
\|f'-\mathsf P_J\|_{\infty,J}
\le
(1+\Lambda_q)\|f'-R_J\|_{\infty,J}
\le a_r|J|^\alpha,
\]
where
$a_r:=\max\left\{1,\frac{1+\Lambda_q}{(r-1)!}\right\}$.
This argument also covers $r=1$, since the two-node interpolant
reproduces the constant polynomial $R_J$.

Set
$h_\delta:=\left(\frac{\delta}{a_r}\right)^{\frac{1}{\alpha}}$
and $n:=\left\lceil\frac{|I|}{h_\delta}\right\rceil$.
Partition $I$ into $n$ equal intervals and construct
$\mathsf P_J$ on each interval using derivative queries.
Adjacent interpolants agree with $f'$ at their common endpoint,
so the resulting piecewise polynomial $P_\delta$ is continuous.
The same endpoint interpolation gives
$P_\delta(a)=f'(a)$ and $P_\delta(b)=f'(b)$.
Every mesh interval has length at most $h_\delta$, and hence
$\|f'-P_\delta\|_{\infty,I}
\le a_rh_\delta^\alpha=\delta$.

Finally, each interval requires at most $q$ derivative queries.
The total number is therefore at most
\[
qn
\le
q\left(
1+a_r^{\frac{1}{\alpha}}|I|
\delta^{-\frac{1}{\alpha}}
\right)
=
O_{r,\beta_r}\!\left(
1+|I|\delta^{-\frac{1}{\nu-1}}
\right).
\]
All query locations are deterministic.
The Taylor polynomials are used only in the error analysis;
the construction queries only $f'$ at the interpolation nodes.
\end{proof}

\section{Proofs in Section~\ref{sec:6}}

\subsection{Proof of Lemma~\ref{lmm:randomized-hard-family}}
\label{app:randomized-hard-family-properties}

\begin{proof}
The fixed kernel gives $\Phi\in C^\infty(\mathbb R)$,
$0\le\Phi\le1$, and $\Phi^{(q)}=0$ outside $(0,1)$ for every $q\ge1$.
In particular, all positive-order derivatives of $\Phi$ are bounded.
Choose
\[
K_{p,\boldsymbol\beta}\ge
\max\left\{1,
\max_{1\le r\le p}2\|\Phi^{(r)}\|_\infty,
\max_{1\le r\le p}\|\Phi^{(r+1)}\|_\infty\right\}.
\]
The accuracy assumption gives $m\ge2$, while the definition of $m$
gives $\delta m=1$ and $\delta\ge2K_{p,\boldsymbol\beta}\varepsilon$.

\noindent\textbf{Regularity and initial gap.}
Since $1-j\le0$ and $m-j+1\ge1$, the endpoint values on the base period are
$f_j(0)=f_j(m)=\frac1{K_{p,\boldsymbol\beta}}$.
Moreover, $f_j'(0)=f_j'(m)=-\frac{\delta}{K_{p,\boldsymbol\beta}}$
and $f_j^{(q)}(0)=f_j^{(q)}(m)=0$ for every $q\ge2$.
Thus the periodic extension matches in every derivative at each multiple
of $m$. The same values and derivatives match the linear extension at $0$,
so $f_j\in C^\infty(\mathbb R)$.

On $[0,m]$, both $1-\delta x$ and $\Phi(x-j+1)$ are nonnegative and
on $(-\infty,0]$, $f_j(x)\ge\frac1{K_{p,\boldsymbol\beta}}$.
Periodicity therefore gives $f_j\ge0$ on $\mathbb R$, and hence
$f_j(0)-\inf f_j\le\frac1{K_{p,\boldsymbol\beta}}\le1$.
Also, $f_j'(0)=-\frac{\delta}{K_{p,\boldsymbol\beta}}\le-2\varepsilon$.

Fix $1\le r\le p$.
The nonconstant part of $f_j^{(r)}$ consists of translates of
$\frac{\Phi^{(r)}}{K_{p,\boldsymbol\beta}}$ on the disjoint hidden intervals
and is zero elsewhere. Its oscillation is therefore at most
$\frac{2\|\Phi^{(r)}\|_\infty}{K_{p,\boldsymbol\beta}}$.
Smooth gluing also gives
$\|f_j^{(r+1)}\|_\infty\le
\frac{\|\Phi^{(r+1)}\|_\infty}{K_{p,\boldsymbol\beta}}$.
For $d:=|x-y|>0$, these two bounds imply
\[
\begin{aligned}
|f_j^{(r)}(x)-f_j^{(r)}(y)|
&\le\min\left\{
\frac{2\|\Phi^{(r)}\|_\infty}{K_{p,\boldsymbol\beta}},
\frac{\|\Phi^{(r+1)}\|_\infty}{K_{p,\boldsymbol\beta}}d
\right\}\\
&\le\min\{1,d\}\le d^{\beta_r}.
\end{aligned}
\]
The last inequality follows from $d\le d^{\beta_r}$ when $d\le1$
and $1\le d^{\beta_r}$ when $d\ge1$.
Dividing by $d^{\beta_r}$ and taking the supremum gives
$[f_j^{(r)}]_{\beta_r}\le1$.

\noindent\textbf{Indistinguishability.}
For $x\le0$, all functions have the same linear jet, so the oracle response
is independent of $j$.
For $x>0$, let $t:=x-m\lfloor x/m\rfloor\in[0,m)$ and
$\ell:=1+\lfloor t\rfloor\in[m]$.
If $j>\ell$, then $t-j+1\le0$; if $j<\ell$, then $t-j+1\ge1$.
Consequently,
\[
f_j(x)=\frac1{K_{p,\boldsymbol\beta}}
\begin{cases}
1-\delta t, & j>\ell,\\
2-\delta t, & j<\ell,\\
1-\delta t+\Phi(t-\ell+1), & j=\ell.
\end{cases}
\]
For $j\ne\ell$, the positive-order derivatives of the step vanish, giving
$f_j'(x)=-\frac{\delta}{K_{p,\boldsymbol\beta}}$ and
$f_j^{(r)}(x)=0$ for $2\le r\le p$.
Thus all indices $j>\ell$ give one common full response, and all indices
$j<\ell$ give another. The single remaining index $j=\ell$ contributes
at most one additional response.
At an integer residue, the flat endpoint derivatives make this transition
response coincide with a linear response.
Hence every fixed query has at most three possible full responses,
including at the interval and period boundaries.

\noindent\textbf{Localization of stationary points.}
Outside the open copies $(j-1+km,j+km)$, $k\ge0$, the step derivative
vanishes, including at their endpoints. The negative linear extension
has the same derivative, so every such point satisfies
$f_j'(x)=-\frac{\delta}{K_{p,\boldsymbol\beta}}\le-2\varepsilon$.
It therefore cannot be $\varepsilon$-stationary.

Finally, these open unit intervals have integer left endpoints.
If copies associated with $j$ and $j'$ intersect, their left endpoints
must agree, giving $j-j'=(k'-k)m$.
Since $|j-j'|\le m-1$, this implies $j=j'$.
Thus a stationary point lies in copies associated with exactly one index.
\end{proof}
\subsection{Proof of Theorem~\ref{thm:randomized-full-jet-lower-bound}}
\label{app:randomized-full-jet-lower-bound}

\begin{proof}
Fix $K_{p,\boldsymbol\beta}$ as in Lemma~\ref{lmm:randomized-hard-family}.
Choose the theorem constant sufficiently small that
$c_{p,\boldsymbol\beta}\le\frac1{64K_{p,\boldsymbol\beta}^2}$
and $c_{p,\boldsymbol\beta}\le\frac1{2\log3}$.

\noindent\textbf{Normalized parameters.}
Suppose $L_1=\cdots=L_p=\Delta=1$ and
$0<\varepsilon\le c_{p,\boldsymbol\beta}$.
Then $\varepsilon\le\frac1{4K_{p,\boldsymbol\beta}}$, so the hard family
is admissible by Lemma~\ref{lmm:randomized-hard-family}.
Since $\frac1{2K_{p,\boldsymbol\beta}\varepsilon}\ge2$, its size satisfies
\[
m=\left\lfloor\frac1{2K_{p,\boldsymbol\beta}\varepsilon}\right\rfloor
\ge\frac1{2K_{p,\boldsymbol\beta}\varepsilon}-1
\ge\frac1{4K_{p,\boldsymbol\beta}\varepsilon}.
\]

Consider any randomized algorithm using at most $N$ queries per execution
and succeeding with probability at least $\frac12$ on every admissible input.
Let $J$ be uniform on $[m]$, independently of the algorithm's random seed $U$,
and run the algorithm on $f_J$.
Fix the entire random seed, including any randomness used to choose the output.
Each query is then determined by the preceding replies, and each query node
has at most three children by the indistinguishability property.

There are at most $3^N$ output leaves.
Indeed, a depth-zero tree has at most one output leaf, and each of the
at most three subtrees of a query node has remaining depth at most $N-1$.
Induction gives the bound. Stopping before depth $N$ is allowed in this
argument, and a nonterminating branch contributes no output leaf.
Each output leaf has a fixed point. By localization, that point is
$\varepsilon$-stationary for at most one index $j$, whether or not it was
queried. Thus, for each fixed seed, the algorithm succeeds on at most
$3^N$ members of the family.

Since $J$ is uniform and independent of $U$, averaging over the seed bounds
the success probability by $\frac{3^N}{m}$.
The success guarantee on every $f_j$ bounds the same probability below by
$\frac12$. Therefore,
\[
\frac12
\le\mathbb P\{\text{success}\}
\le\frac{3^N}{m},
\qquad
N\ge\frac{\log(m/2)}{\log3}
\ge\frac{\log(1/\varepsilon)-\log(8K_{p,\boldsymbol\beta})}{\log3}.
\]
Our choice of $c_{p,\boldsymbol\beta}$ ensures
$\log(1/\varepsilon)\ge2\log(8K_{p,\boldsymbol\beta})$.
Consequently, $N\ge\frac1{2\log3}\log(1/\varepsilon)$,
and the choice of $c_{p,\boldsymbol\beta}$ gives
$N\ge c_{p,\boldsymbol\beta}\log(1/\varepsilon)$.

\noindent\textbf{General parameters.}
Choose $r_*$ attaining the minimum of
$L_r^{\frac1{\nu_r}}\Delta^{\frac{\nu_r-1}{\nu_r}}$
over $1\le r\le p$, and set
$s:=\left(\frac{\Delta}{L_{r_*}}\right)^{\frac1{\nu_{r_*}}}$
as in Lemma~\ref{lmm:normal}.
For every $1\le r\le p$, the normalized bounds satisfy
\[
\widehat L_r=\frac{L_rs^{\nu_r}}{\Delta}
=\left(
\frac{L_r^{\frac1{\nu_r}}\Delta^{\frac{\nu_r-1}{\nu_r}}}
{L_{r_*}^{\frac1{\nu_{r_*}}}
\Delta^{\frac{\nu_{r_*}-1}{\nu_{r_*}}}}
\right)^{\nu_r}\ge1.
\]
Hence the normalized class contains the unit-parameter hard family.
Its normalized accuracy is
\[
\widehat\varepsilon=\frac{s\varepsilon}{\Delta}
=\frac{\varepsilon}
{\displaystyle\min_{1\le r\le p}
L_r^{\frac1{\nu_r}}\Delta^{\frac{\nu_r-1}{\nu_r}}}
\le c_{p,\boldsymbol\beta}.
\]
Any algorithm correct on the normalized class must therefore satisfy
the preceding lower bound at accuracy $\widehat\varepsilon$.
Lemma~\ref{lmm:normal} preserves the query count and success probability,
so the original class also requires at least
$c_{p,\boldsymbol\beta}\log(1/\widehat\varepsilon)$ queries.
Substituting the displayed accuracy proves the theorem.
\end{proof}

\section{Proofs in Section~\ref{sec:7}}
\subsection{Proof of Lemma~\ref{lmm:value-interpolation}}
\label{app:value-interpolation}

\begin{proof}
Write $\nu:=r+\beta_r$ and set $q:=\max\{r+1,3\}$.
Construct $\mathsf Q_I$ by interpolation at $q$ equally spaced
nodes of $I$, including both endpoints.
Let $I=[a,b]$, $H:=b-a$, and let $R$ be the degree-$r$
Taylor polynomial of $f$ at $a$.
Taylor's integral remainder and
$|f^{(r)}(t)-f^{(r)}(a)|\le H^{\beta_r}$ give
\[
\|f-R\|_{\infty,I}\le\frac{H^\nu}{r!},
\qquad
\|f'-R'\|_{\infty,I}
\le\frac{H^{\nu-1}}{(r-1)!}.
\]
For the first estimate, the remainder is
$\frac{1}{(r-1)!}\int_a^x
(x-t)^{r-1}(f^{(r)}(t)-f^{(r)}(a))\,dt$.
The second follows from the analogous formula for $f'$ when $r\ge2$;
when $r=1$, it is precisely the H\"older bound on $f'$.

Let $\ell_1,\ldots,\ell_q$ be the Lagrange basis polynomials
for the chosen nodes on $[0,1]$, and set
$\Gamma_q:=\sum_{j=1}^q\|\ell_j'\|_{\infty,[0,1]}$.
Since $q-1\ge r$, interpolation reproduces $R$.
After rescaling $I$ to $[0,1]$, the derivative of the interpolation
operator has norm at most $\frac{\Gamma_q}{H}$.
Consequently,
\[
\begin{aligned}
\|\mathsf Q_I'-f'\|_{\infty,I}
&\le \|R'-f'\|_{\infty,I}
   +\|(\mathsf Q_I-R)'\|_{\infty,I}\\
&\le \frac{H^{\nu-1}}{(r-1)!}
   +\frac{\Gamma_q}{H}\|f-R\|_{\infty,I}\\
&\le \left(\frac{1}{(r-1)!}
   +\frac{\Gamma_q}{r!}\right)H^{\nu-1}.
\end{aligned}
\]
Thus $\|f'-\mathsf Q_I'\|_{\infty,I}\le a_r|I|^{\nu-1}$ holds with
$a_r:=\max\{1,\frac{1}{(r-1)!}+\frac{\Gamma_q}{r!}\}$.
Endpoint interpolation gives $\mathsf Q_I(a)=f(a)$ and
$\mathsf Q_I(b)=f(b)$.
The construction uses at most $q=O_r(1)$ function-value queries.
\end{proof}

\subsection{Proof of Theorem~\ref{thm:deterministic-zeroth-order-upper}}
\label{app:function-value-upper}

\begin{proof}
Fix $r\in\mathcal I$ and write $\nu:=r+\beta_r$.
We first assume $[f^{(r)}]_{\beta_r}\le1$,
$f(0)-\inf f\le1$, and $0<\varepsilon\le1$.

\noindent\textbf{Derivative estimates.}
Lemma~\ref{lmm:value-interpolation} on an interval of length
$h_\varepsilon$ gives $|D(x)-f'(x)|\le\frac{\varepsilon}{8}$
at either endpoint.
It also gives the required uniform bound for the terminal interpolant.

Whenever $|D(x)|\le\frac{3\varepsilon}{4}$, the returned point satisfies
$|f'(x)|\le\frac{7\varepsilon}{8}<\varepsilon$.
Otherwise, $D(x)<-\frac{3\varepsilon}{4}$ implies
$f'(x)<-\frac{5\varepsilon}{8}$, and
$D(x)>\frac{3\varepsilon}{4}$ implies
$f'(x)>\frac{5\varepsilon}{8}$.
If the initial estimate is positive, reflection negates both
$f'(0)$ and its cached estimate, preserving their error bound.
It also preserves the H\"older condition and initial gap.
The reflected interpolation interval still has the origin as an endpoint.

\noindent\textbf{Interval invariant.}
Every active interval has
$D(a)<-\frac{3\varepsilon}{4}$, both endpoint estimates lie outside
$[-\frac{3\varepsilon}{4},\frac{3\varepsilon}{4}]$, and
at least one of the certificates
\[
D(b)>\frac{3\varepsilon}{4},
\qquad\text{or}\qquad
\frac{f(b)-f(a)}{b-a}\ge-\frac{\varepsilon}{4}.
\]
Initially, $a=0$ and $b=T:=\frac{4}{\varepsilon}$.
The gap bound gives $f(T)-f(0)\ge-1$, hence the secant certificate.
The estimate at $T$ either returns a stationary point or lies outside
the acceptance range, establishing the invariant.

Let $c:=\frac{a+b}{2}$.
An accepted estimate returns a valid point as proved above.
If $D(c)>\frac{3\varepsilon}{4}$, replacing $b$ by $c$
establishes the first certificate and preserves the negative left estimate.
Otherwise, $D(c)<-\frac{3\varepsilon}{4}$.
If $D(b)>\frac{3\varepsilon}{4}$, replacing $a$ by $c$
preserves both the first certificate and the negative left estimate.

In the remaining case, both $D(c)$ and $D(b)$ are negative.
The parent interval must therefore have the secant certificate.
The parent secant is the average of the two half-interval secants,
so at least one half has secant slope at least
$-\frac{\varepsilon}{4}$.
Since $c-a=b-c$, the left half has the larger slope exactly when
$2f(c)\ge f(a)+f(b)$.
The final two branches select that half, whose left endpoint is either
$a$ or $c$ and has a negative estimate.
Thus every nonreturning iteration preserves the invariant and halves
the interval length.

\noindent\textbf{Terminal output.}
Let $P:=\mathsf Q_{[a,b]}'$.
Since $b-a\le h_\varepsilon$, we have
$\|P-f'\|_{\infty,[a,b]}\le\frac{\varepsilon}{8}$.
The negative left estimate gives
$P(a)\le f'(a)+\frac{\varepsilon}{8}<-\frac{\varepsilon}{2}$.
If the first certificate holds, then
$P(b)\ge f'(b)-\frac{\varepsilon}{8}>\frac{\varepsilon}{2}$.
If the secant certificate holds, endpoint interpolation gives
\[
\frac{1}{b-a}\int_a^b P(x)\,dx
=\frac{\mathsf Q_{[a,b]}(b)-\mathsf Q_{[a,b]}(a)}{b-a}
=\frac{f(b)-f(a)}{b-a}
\ge-\frac{\varepsilon}{4}.
\]
Hence $P$ attains a value at least $-\frac{\varepsilon}{4}$.
In both cases, continuity gives a point $z\in[a,b]$ with
$P(z)=-\frac{\varepsilon}{2}$.
Every such point satisfies
$|f'(z)|\le|P(z)|+\frac{\varepsilon}{8}
=\frac{5\varepsilon}{8}<\varepsilon$.
The root is computed from the known polynomial without additional
oracle queries.

\noindent\textbf{Query budget.}
Lemma~\ref{lmm:value-interpolation} gives $O_r(1)$ value queries
for each derivative estimate and for the terminal interpolant.
The loop makes at most
$\left\lceil\log_2\left(\frac{4/\varepsilon}{h_\varepsilon}\right)\right\rceil$
iterations, in addition to the two initial estimates.
The ratio in the logarithm equals
$4(8a_r)^{\frac{1}{\nu-1}}\varepsilon^{-\frac{\nu}{\nu-1}}$,
so the total is
$O_{r,\beta_r}(1+\log(1/\varepsilon))$.
Cached endpoint values suffice for every secant comparison.

\noindent\textbf{General parameters.}
Lemma~\ref{lmm:normal} gives normalized accuracy
$\widehat\varepsilon:=\frac{\varepsilon}
{L_r^{\frac{1}{\nu}}\Delta^{\frac{\nu-1}{\nu}}}$.
Run the normalized algorithm with accuracy
$\min\{\widehat\varepsilon,1\}$.
Its output also meets accuracy $\widehat\varepsilon$.
Each normalized function-value query is simulated by one original
function-value query, and rescaling preserves the stationarity guarantee.
Thus, for every selected $r$, the query bound is
\[
O_{r,\beta_r}\left(
1+\log_+\left(
\frac{L_r^{\frac{1}{\nu}}\Delta^{\frac{\nu-1}{\nu}}}{\varepsilon}
\right)\right).
\]
Taking the maximum of the constants over $1\le r\le p$ and choosing
an index minimizing
$L_r^{\frac{1}{\nu_r}}\Delta^{\frac{\nu_r-1}{\nu_r}}$
proves the theorem.
\end{proof}
\end{document}